\documentclass[preprint,12pt]{elsarticle}

\usepackage{amsfonts,amssymb,amsmath,amsthm,bm}
\usepackage{mathtools}
\usepackage{dsfont}
\usepackage{color}
\usepackage{geometry}                		
\usepackage{subfigure}

\newtheorem{lemma}{Lemma}
\newtheorem{theorem}{Theorem}
\newtheorem{corollary}{Corollary}
\newtheorem{remark}{Remark}

\makeatletter
\def\ps@pprintTitle{%
  \let\@oddhead\@empty
  \let\@evenhead\@empty
  \let\@oddfoot\@empty
  \let\@evenfoot\@oddfoot
}
\makeatother

\begin{document}

\begin{frontmatter}

\title{The Feller diffusion as the limit of a coalescent point process}
\author[label1]{Conrad J.\ Burden}
\ead{conrad.burden@anu.edu.au}
\author[label2]{Robert C.\ Griffiths}
\ead{Bob.Griffiths@Monash.edu}
\address[label1]{Mathematical Sciences Institute, Australian National University, Canberra, Australia}
\address[label2]{School of Mathematics, Monash University, Australia}


\begin{abstract}
The Feller diffusion is studied as the limit of a coalescent point process in which the density of the node height distribution is 
skewed towards zero.  Using a unified approach, a number of recent results pertaining to scaling limits of branching processes are 
reviewed and reinterpreted as properties of the Feller diffusion arising from this limit.  The notion of Bernoulli sampling of a finite population is 
extended to the diffusion limit to cover finite Poisson-distributed samples drawn from infinite continuum populations.  We show that 
the coalescent tree of a Poisson-sampled Feller diffusion corresponds to a coalescent point process with a node height distribution 
taking the same algebraic form as that of a Bernoulli-sampled birth-death process.  By adapting methods for analysing $k$-sampled 
birth-death processes, in which the sample size is pre-specified, we develop methods for studying the coalescent properties of the 
$k$-sampled Feller diffusion.   
\end{abstract}

\end{frontmatter}

%

\section{Introduction}
\label{sec:Introduction}

The Feller diffusion~\citep{Feller1939,feller1951diffusion} is defined as the stochastic process $\big(X(t)\big)_{t \in \mathbb{R}_{\ge 0}}$ 
with generator 
\begin{equation}	\label{FellerGenerator}
\mathcal{L} = \tfrac{1}{2} x \frac{\partial^2}{\partial x^2}+ \alpha x \frac{\partial}{\partial x}, \qquad x \in \mathbb{R}_{\ge 0}, \quad \alpha \in \mathbb{R}. 
\end{equation}
It arises as the limit of a continuous-time, finite-population linear birth-death (BD) process in 
which the birth and death rates become infinite, but their difference remains finite~\citep[Section~2.1]{BurdenGriffiths24}.  This limit has implicitly manifested 
several times in the literature in the guise of the large-population limit of a critical BD process \citep[Section~3.2]{Aldous05}, the long-time, near-critical limit 
of BD processes~\citep{o1995genealogy,Harris20}, the large-population limit of sampled BD processes \citep{Crespo21}, and the diffusion limit of a 
Bienaym\'{e}-Galton-Watson process \citetext{\citealp[p235 and Section~3.7]{cox78}; \citealp{burden2016genetic}; \citealp{BurdenSoewongsoso19}}.

The purpose of this paper is to survey and extend recently reported results associated with scaling limits of BD processes, which we reinterpret as 
Feller diffusions conditioned on a single ancestral founder~\citep{BurdenGriffiths25}.  We do so under a unified set of principles based on a coalescent 
point process (CPP) approach~\citep{Popovic04,Aldous05}.  The CPP and its relationship to coalescent trees is reviewed in Section~\ref{sec:CPPs}, with emphasis on general properties 
which are needed later in the paper.  We also observe a previously unrecognised close connection between the distribution of coalescent times for a process 
initiated at a specified time $t$ in the past and for a process with an assumed uniform prior on an initiation time in the interval $[0, t]$ in the past.  

The CPP approach to the linear BD process is summarised in Section~\ref{sec:BDProcess}, and a convenient parameterisation due to \citet{Wiuf18} is introduced which 
makes explicit a symmetry between super-critical and sub-critical process.  This symmetry is shown to emerge naturally from the CPP approach as a consequence 
of requiring that the ancestral tree of a final population converges to a single ancestral founder in the sub-critical case.  In Section~\ref{sec:FellerCPP} the Feller 
diffusion is shown to correspond to a limiting case of a CPP in which the density of the node-height distribution is skewed towards zero.  A concise derivation of the 
known joint density of coalescent times for a Feller diffusion descended from a single founding ancestor and conditioned on an observed current population size 
is given.  

Section~\ref{sec:Sampling} is devoted to sampling distributions.  Both Bernoulli sampling, in which each member of a population is sampled with a specified 
probability, and $k$-sampling, in which a specified number of individuals are sampled, are considered.   We extend to the Feller diffusion a result due to 
\citet{Wiuf18} that a Bernoulli-sampled BD process corresponds to a CPP with a node-height density taking the same functional form as that of an 
unsampled BD process.  We also define an analogue of Bernoulli sampling appropriate for Feller diffusions, which we call Poisson sampling.  By extending a procedure 
due to \citet{Lambert18} we demonstrate how to obtain $k$-sampling distributions for a Feller diffusion from Poisson-sampling distributions.  The procedure 
is then used to reproduce known joint distributions of coalescent times for near-critical BD processes due to \citet{Harris20}.  A method due to \citet{Crespo21} 
for determining $k$-sampled distributions in the large-population limit of a BD process is shown to apply to the Feller diffusion, and it is noted that the method 
may have broader applications.  

In Section~\ref{sec:ExpectedWkWiuf} a formula for calculating expected inter-coalescent waiting times in terms of the hypergeometric function is given for 
CPPs whose node-height distribution takes the form of the \citet{Wiuf18} parameterisation.  Conclusions are drawn in Section~\ref{sec:Conclusions}.  
\ref{sec:notation} contains a table of notation.  

%

\section{Coalescent point processes}
\label{sec:CPPs}

The CPP is neatly summarised in the introduction to \citet{Lambert13b} as follows.  
\begin{quote}
 A CPP with stem age $t$ is a random ultrametric tree with height $t$ whose node depths are characterised by independent draws 
 from the same distribution with range $[0, \infty)$ until a value larger than $t$ is drawn.  Let the first draw bigger than $t$ be the
 $n$-th draw.  The first $n - 1$ draws $H_1, \ldots ,H_{n - 1}$ give rise to a phylogenetic tree on $n$ tips.  (See Fig.~\ref{fig:CPP_diagram}.)
\end{quote}
In this definition, the parameter $t$ can be identified with the root of an associated coalescent tree.  
The $k$-ancestor to $(k - 1)$-ancestor coalescent event times $T_k$ measured back from the from the present are related to the stem height $t$ 
and to order statistics $H_{(1)} < \cdots < H_{(n - 1)}$ of the node heights by 
\begin{equation}	\label{OrderedHkToTk}
T_1 = t, \qquad T_k = H_{(n + 1 - k)}, \quad k = 2, \ldots, n.  
\end{equation}

\begin{figure}
 \centering
 \includegraphics[width=\linewidth]{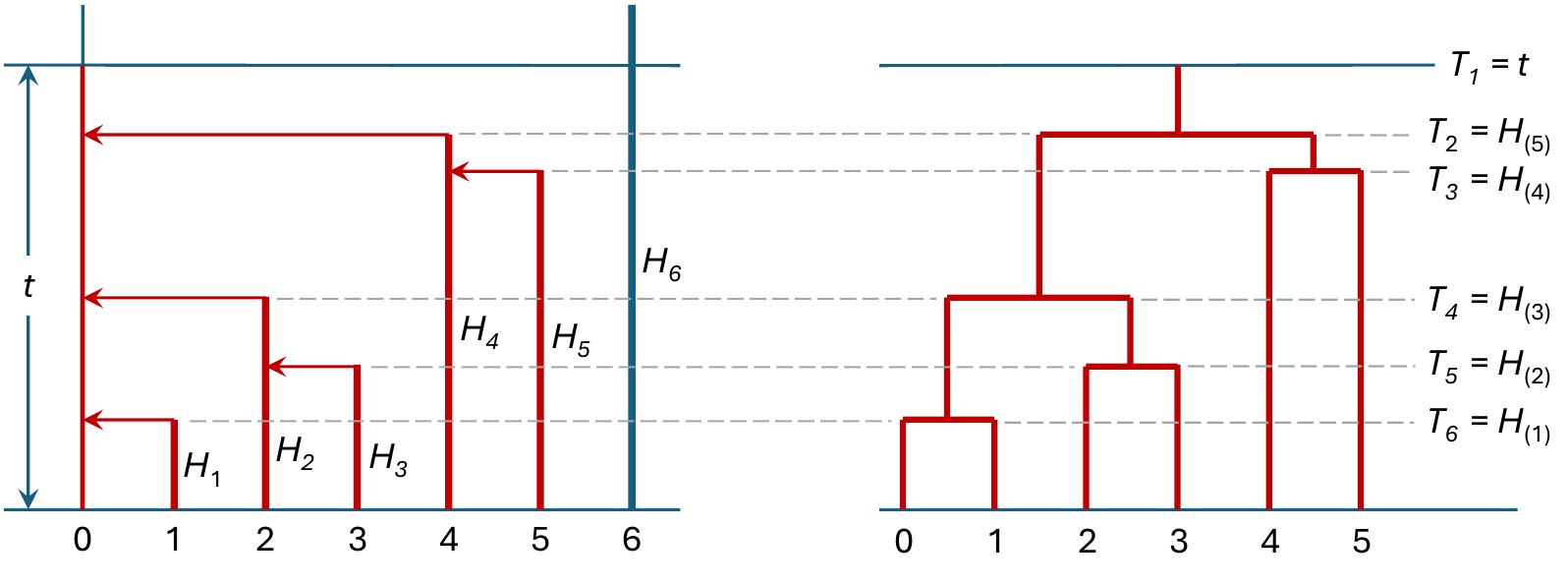}
  \caption{A CPP with $n = 6$ and the associated ultrametric tree, interpreted as a coalescent tree.}
 \label{fig:CPP_diagram}
 \end{figure}

\citet{Lambert13b} show that a CPP can be associated with the coalescent tree of any BD process $\big(N(u)\big)_{u \in [0, t]}$, 
for which the forward-in-time birth and death rates (or speciation and extinction rates in the context of phylogenetics), $\lambda(u)$ 
and $\mu(u)$ respectively, depend on the absolute time $u$ since initiation of the process, are independent of $N(u)$, and the death rate 
$\mu(u)$ possibly further depends on a non-heritable trait\footnote{See \citet[Section~4.4]{Lambert18} for the definition of a non-heritable trait.}.  
From Section~\ref{sec:BDProcess} onwards the current paper is restricted to the linear BD process, that is, the BD process with constant, 
trait-independent rates $\lambda$ and $\mu$, and to its diffusion limit, the Feller diffusion with constant growth rate $\alpha$.  
The remainder of Section~\ref{sec:CPPs} is devoted to properties of coalescent trees associated with CPPs in general. 

\citet[Section~1.2]{Lambert18} introduces the terminology {\em inverse tail distribution} for the function 
\[
F^{\rm IT}_H(\tau) := \frac{1}{1 - F_H(\tau)}, 
\]
where $F_H(\tau) := \mathbb{P}(H \le \tau)$ is the typical node height cumulative distribution function ({\em cdf}).  
This inverts to $F_H(\tau) = 1 - {F^{\rm IT}_H(\tau)}^{-1}$, and hence the {\em coalescent density} or {\em node height density} is 
\[
f_H(\tau) := F_H'(\tau) = \frac{{F^{\rm IT}_H}'(\tau)}{{F^{\rm IT}_H(\tau)}^2},  
\]
where prime indicates the derivative.  

\citet[Section~1.2]{Ignatieva20} introduce the idea of a {\em reversed reconstructed process} (RRP) to construct coalescent trees.  
The terminology {\em reconstructed process} was introduced by \citet{Nee94} to reconstruct the phylogenetic history of a set of observed contemporary species 
when no information about extinct lineages is known.  Given a BD process which is stopped at time $t$ (the ``present''), 
the reconstructed process is a forward-in-time inhomogeneous pure-birth process with rate $\lambda_{\rm eff}(u)$, say, which generates the subtree with 
extinct branches pruned, conditional on survival of at least one lineage to time~$t$.  The RRP is the corresponding inhomogeneous 
pure-death process with rate $\mu_{\rm eff}(\tau) = \lambda_{\rm eff}(t - \tau)$, $\tau > 0$, which runs backwards in time from the present.  In the language 
of population genetics it generates a coalescent tree.  The  {\em cdf} of the typical CPP node height is related to the RRP effective death rate by 
\[
F_H(\tau) := \mathbb{P}(H \le \tau) = 1 - \exp\left\{-\int_0^\tau \mu_{\rm eff}(\xi) d\xi \right\},  
\]
 which inverts to 
 \begin{equation}	\label{FHToMuEff}
 \mu_{\rm eff}(\tau) = \frac{d}{d\tau} \log\left\{(1 - F_H(\tau))^{-1}\right\}.  
 \end{equation}
 
%

\subsection{Coalescent time distribution given $T_1 = t$}
\label{TkGivenT1}

It is useful to condition a CPP on a fixed number of tips, $n$, in which case the process generating the associated coalescent tree is called 
a {\em conditioned reconstructed process}~\citep{Gernhard08}.  For such a process, \citet[Section~2]{Stadler09} distinguishes between the 
topological structure of the tree, which she refers to as the discrete {\em ranked oriented tree}, and the set of coalescent times $T_k$.  From the CPP 
construction, it is evident that the ranked oriented tree and joint coalescent time distributions are independent, and that each ranked oriented tree 
has equal probability~\citep[Theorem~2.3]{Gernhard08}.  Although these properties were first demonstrated for the linear BD process in particular, 
they hold for trees associated with CPPs in general.  

The probability that a CPP with stem age $t$ generates a tree with $n$ leaves is 
$
F_H(t)^{n - 1} (1 - F_H(t)), 
$
and the joint probability/density that there are $n$ leaves and node heights $h_1, \ldots, h_{n - 1}$ is 
$
\left(\prod_{j = 1}^{n - 1} f_H(h_j)\right) (1 - F_H(t)).  
$
Thus the density of node heights given $n$ leaves is 
\[
f_{H_1, \ldots, H_{n - 1}}(h_1, \ldots, h_{n - 1} \mid n) = \frac{\left(\prod_{j = 1}^{n - 1} f_H(h_j)\right) (1 - F_H(t))}{F_H(t)^{n - 1} (1 - F_H(t))}
			= \prod_{j = 1}^{n - 1} \left(\frac{f_H(h_j)}{F_H(t)}\right).  
\]
It follows that, conditional on $n$ leaves, the node heights are identically and independently distributed (i.i.d.) with common density 
$f_H(h)/F_H(t)$, $h \in [0, t]$.  From Eq.~(\ref{OrderedHkToTk}) we then have the distribution of coalescent times of the associated tree 
in terms of the distribution of order statistics of i.i.d.\ node heights.  

Suppressing the dependence on $t$, for convenience set 
\begin{equation}	\label{normalisedFH}
F(\tau) = \frac{F_H(\tau)}{F_H(t)}, \qquad f(\tau) = \frac{f_H(\tau)}{F_H(t)}, \qquad \tau \in [0, t].  
\end{equation}
Then the joint {\em cdf} of coalescent times in a tree generated by a CPP conditioned to have $n$ leaves is
\begin{eqnarray}	\label{T2_to_Tn_cdf}
F_{T_2, \ldots, T_n \mid T_1 = t}(\tau_2, \ldots, \tau_n \mid n) & := & \mathbb{P}(T_2 \le \tau_2, \ldots, T_n \le \tau_n \mid T_1 = t, n \text{ leaves})
														\nonumber \\
	& = & F_{H_{(1)}, \ldots, H_{(n - 1)}}(\tau_k, \ldots, \tau_2 \mid n) \nonumber \\
	& = & (n - 1)! \prod_{j = 2}^n F(\tau_j), \nonumber \\ 
	& & \qquad  n \ge 2, \quad t > \tau_2 > \cdots > \tau_n > 0, 
\end{eqnarray}
and the joint density of the $k - 1$ largest coalescent times is 
\begin{eqnarray}	\label{T2_to_Tk_density}
f_{T_2, \ldots, T_k \mid T_1 = t}(\tau_2, \ldots, \tau_k \mid n) 
	& = & f_{H_{(n - k + 1)}, \ldots, H_{(n - 1)}}(\tau_k, \ldots, \tau_2 \mid n) \nonumber \\
	& = & \frac{(n - 1)!}{(n - k)!} \left(\prod_{j = 2}^k f(\tau_j)\right) F(\tau_k)^{n - k}, \nonumber \\
	& & \qquad  k = 2, \ldots, n, \quad t > \tau_2 > \cdots > \tau_k > 0. 
\end{eqnarray}
The marginal density of the coalescent time $T_k$ is 
\begin{eqnarray}	\label{Tk_density}
f_{T_k \mid T_1 = t}(\tau \mid n) & = & f_{H_{(n - k + 1)}}(\tau \mid n) \nonumber \\ 
	& = & \frac{(n - 1)!}{(n - k)!(k - 2)!} \left(1 - F(\tau)\right)^{k - 2} f(\tau) 
							\left(F(\tau)\right)^{n - k}  \nonumber \\  
	& & \qquad\qquad\qquad\qquad\qquad\qquad  2 \le k \le n, \quad 0 <  \tau < t,  
\end{eqnarray}
the joint marginal density of $T_i$ and $T_j$ is 
\begin{eqnarray*}
f_{T_i, T_j \mid T_1 = t}(\tau, s \mid n) & = & f_{H_{(n - i + 1), (n - j + 1)}}(\tau, s \mid n) \nonumber \\ 
	& = & \frac{(n - 1)!}{(n - j)!(j - i - 1)!(i - 2)!} \times \nonumber \\
	& & \qquad  F(s)^{n - j}(F(\tau) - F(s))^{j - i - 1}(1 - F(\tau))^{i - 2} f(\tau)f(s), \nonumber \\ 
	& & \qquad\qquad\qquad\qquad  2 \le i < j \le n, \quad t > \tau > s > 0,  
\end{eqnarray*}
and hence the density of $T_i$ conditional on $T_j$ for $i < j$ is 
\begin{eqnarray}	\label{Ti_given_Tj}
f_{T_i \mid T_j = s, T_1 = t}(\tau) & = & \frac{f_{T_i, T_j \mid T_1 = t}(\tau, s \mid n)}{f_{T_j \mid T_1 = t}(\tau \mid n)} \nonumber \\ 
	& = & \frac{(j - 2)!}{(i - 2)!(j - i - 1)!} \frac{(F(\tau) - F(s))^{j - i - 1}(1 - F(\tau))^{i - 2}}{(1 - F(s))^{j - 2}} f(\tau)  \nonumber \\
	& & \qquad\qquad\qquad\qquad  2 \le i < j \le n, \quad t > \tau > s > 0,  
\end{eqnarray}
which is independent of $n$.  

%

\subsection{Coalescent time distribution with uniform prior on $T_1$} 
\label{sec:T1unifPrior}

As an alternative to fixing the stem height $T_1$ of a branching process {\em a priori}, 
various authors \citep{Aldous05,Gernhard08,Gernhard08a,Stadler09,Wiuf18,Ignatieva20} have considered
setting an improper uniform prior on $[0, \infty)$ for the time since initiation of the process from a single founder.
After setting a uniform prior on $[0, t]$ with density 
\[
f_{T_1}^{{\rm unif}\, [0, t]} (\tau) = \frac{1}{t}, \qquad 0 \le \tau \le t,  
\]
and calculating posterior probabilities, the limit $t \to \infty$ then enables implementation of an improper uniform prior 
as defined by \citet[Section~4.2]{Wiuf18}.  Here we describe a connection between 
between the posterior coalescent properties for a conditioned reconstructed process under the uniform prior on $[0, t]$ and those 
for a conditioned reconstructed process with fixed stem height $t$.  In general, the limit $t \to \infty$ can be taken at the end 
of the day to implement an improper uniform prior.  

The joint posterior probability that $T_1 \in (\tau, \tau + d\tau)$ and that the tree associated with the CPP has $n$ leaves is, 
in terms of the node heights $H_1, H_2, \ldots$, 
\begin{eqnarray*}
\lefteqn{\mathbb{P}^{{\rm unif}\, [0, t]}(T_1 \in (\tau, \tau + d\tau), n \text{ leaves})} \nonumber \\ 
	& = & \mathbb{P}(H_1, \ldots, H_{n - 1} < \tau, H_n \in (\tau, \tau + d\tau)) 
					\times \mathbb{P}^{{\rm unif}\, [0, t]}(T_1 \in (\tau, \tau + d\tau)) \nonumber \\ 
	& = & \frac{1}{t} F_H(\tau)^{n - 1} f_h(\tau) d\tau.  
\end{eqnarray*}
Condition on $n$ by dividing out the marginal probability that there  are $n$ leaves, 
\begin{eqnarray*}
\lefteqn{\mathbb{P}^{{\rm unif}\, [0, t]}(T_1 \in (\tau, \tau + d\tau) \mid n \text{ leaves})} \nonumber \\ 
	& = & \frac{\frac{1}{t} F_H(\tau)^{n - 1} f_h(\tau) d\tau}{\frac{1}{t} \int_0^tF_H(\xi)^{n - 1} f_h(\xi) d\xi} \nonumber \\
	& = & \frac{F_H(\tau)^{n - 1} f_h(\tau) d\tau}{F_H(t)^n \int_0^1u^{n - 1}du} \nonumber \\
	& = & n \left( \frac{F_H(\tau)}{F_H(t)} \right)^{n - 1} \frac{f_H(\tau)}{F_H(t)} d\tau,   
\end{eqnarray*}
giving the density   
\[
f_{T_1}^{{\rm unif}\, [0, t]} (\tau \mid n) = n F(\tau)^{n - 1} f(\tau), \quad 0 < \tau < t,  
\]
with $F(\tau)$ and $f(\tau)$ as in Eq.~(\ref{normalisedFH}).  
Since, by Eq.~(\ref{T2_to_Tn_cdf}), $T_1$ determines the density of $T_2, \ldots, T_n$, we have the joint density 
\begin{eqnarray*}
\lefteqn{ f_{T_1, \ldots, T_n}^{{\rm unif}\, [0, t]} (\tau_1, \ldots, \tau_n \mid n)} \nonumber \\
	& = & f_{T_2, \ldots, T_n \mid T_1 = \tau_1} (\tau_2, \ldots, \tau_n \mid n) f_{T_1}^{{\rm unif}\, [0, t]} (\tau_1 \mid n)
															\nonumber \\ 
	& = & (n - 1)! \left(\prod_{j = 2}^n \frac{f_H(\tau_j)}{F_H(\tau_1)}\right)  
					\times n \left( \frac{F_H(\tau_1)}{F_H(t)} \right)^{n - 1} \frac{f_H(\tau_1)}{F_H(t)}  \nonumber \\ 
	& = & n! \left(\prod_{j = 1}^n \frac{f_H(\tau_j)}{F_H(t)}\right),    \nonumber \\
	& = & n! \left(\prod_{j = 1}^n f(\tau_j)\right),  \quad t > \tau_1 > \cdots > \tau_n > 0,  
\end{eqnarray*}
using Eq.~(\ref{T2_to_Tk_density}) for the first factor.  Moreover, comparing the last result with Eq.~(\ref{T2_to_Tk_density}), 
\begin{equation}	\label{fgivenT1tofunif}
f_{T_1, \ldots, T_n}^{{\rm unif}\, [0, t]} (\tau_1, \ldots, \tau_n \mid n) 
			= f_{T_2, \ldots, T_{n + 1} \mid T_1 = t} (\tau_1, \ldots, \tau_n \mid n + 1), 
\end{equation}
that is, the distribution of $T_1, \ldots, T_n$ subject to a uniform$[0, t]$ prior on $T_1$ and conditional on $n$ leaves is 
identical to the distribution of $T_2, \ldots, T_{n + 1}$ conditional on $T_1 = t$ and $n + 1$ leaves. Analogous results therefore 
hold for any marginal distribution or expectation value.  For instance, writing $\mathbb{E}_n$ for expectation conditional on $n$ leaves, 
\begin{equation}	\label{EgivenT1toEunif}
\mathbb{E}_n^{{\rm unif}\, [0, t]} [T_k] = \mathbb{E}_{n + 1}[T_{k + 1} \mid T_1 = t].  
\end{equation}

%

\subsection{Iterative formulae for expected coalescent times and waiting times for a general CPP}
\label{sec:iterativeTkWk}

Conditioning on $T_1 = t$ and $n$ leaves, the expected coalescent times satisfy the iterative formula (see \ref{sec:iterativeTandW}) 
\begin{eqnarray}	\label{iterativeTkgivenT1}
\lefteqn{\mathbb{E}_n[T_{k} \mid T_1 = t] = \int_0^t \tau f_{T_k \mid T_1 = t}(\tau \mid n) d\tau} \nonumber \\ 
	& = & \frac{(n - 1)!}{(n - k)!(k - 2)!} \int_0^1 u^{k - 2} (1 - u)^{n - k}G(u) du \nonumber \\
	& = & \frac{n - 1}{k - 2} \mathbb{E}_{n - 1}[T_{k - 1} \mid T_1 = t] - \frac{n - k + 1}{k - 2} \mathbb{E}_n[T_{k - 1} \mid T_1 = t], 
																		\quad 3 \le k \le n, 
\end{eqnarray}
where $G^{-1}(\tau) = 1 - F(\tau)$.  
This enables the conditional expectation of all $T_{k} \mid T_1$ to be computed if the conditional expectation of $T_{2} \mid T_1$
for all $n$ can be calculated.  Note that the conditional expectation of $T_{2} \mid T_1$ depends on the specific form of the 
node height {\em cdf} $F_H(\tau)$.  \citet[Section~8.4]{Wiuf18} gives a similar iterative formula, but with $n$ in place of $n - 1$ in the numerator 
in the first term, which we believe is a mistake.  

Define waiting times between coalescent points for a tree with $n$ leaves by 
\[
W_k = T_ k - T_{k + 1}, \qquad 1 \le k \le n.  
\]
The iterative formula 
\begin{eqnarray}	\label{iterativeWkgivenT1}	
\lefteqn{\mathbb{E}_n[W_{k} \mid T_1 = t]} \nonumber \\ 
	& = & \frac{n - 1}{k - 1} \mathbb{E}_{n - 1}[W_{k - 1} \mid T_1 = t] - \frac{n - k + 1}{k - 1} \mathbb{E}_n[W_{k - 1} \mid T_1 = t], 
																		\quad 2 \le k \le n. 
\end{eqnarray}
then follows from Eq.~(\ref{iterativeTkgivenT1}) (see \ref{sec:iterativeTandW}).  

The expected coalescent times and expected waiting times subject to a uniform$[0, t]$ prior on $T_1$ follow immediately from 
Eq.~(\ref{EgivenT1toEunif}) (see also \citet[Eq.~(18) and Section~8.2 respectively]{Wiuf18} for a direct proof): 
\begin{equation*}	
\mathbb{E}_n^{{\rm unif}\, [0, t]} [T_k]  = 
	\frac{n}{k - 1} \mathbb{E}_{n - 1}^{{\rm unif}\, [0, t]} [T_{k - 1}] - \frac{n - k + 1}{k - 1} \mathbb{E}_n^{{\rm unif}\, [0, t]} [T_{k - 1}], 
																		\quad 2 \le k \le n,  
\end{equation*}
\begin{equation}	\label{iterativeWkunifT1}
\mathbb{E}_n^{{\rm unif}\, [0, t]} [W_k]  = 
	\frac{n}{k} \mathbb{E}_{n - 1}^{{\rm unif}\, [0, t]} [W_{k - 1}] - \frac{n - k + 1}{k} \mathbb{E}_n^{{\rm unif}\, [0, t]} [W_{k - 1}], 
																		\quad 2 \le k \le n.  
\end{equation}

%

\section{Birth-death process as a CPP}
\label{sec:BDProcess}

The reconstructed tree of a linear BD process with constant birth rate $\lambda$ and death rate $\mu$ was recognised as a CPP by \citet{Aldous05} for 
the critical case $\lambda = \mu$ and by \citet{Gernhard08} for general $\lambda$ and $\mu$.  We begin with the following Lemma which rederives the node height 
distribution quoted in \citet[Eq.~(4)]{Lambert18}.  

\begin{lemma}
The reconstructed tree of a BD process $\big(N(u)\big)_{u \in [0, t]}$ with constant birth rate $\lambda$ and death rate $\mu$ 
stopped at time $t$ corresponds to a CPP with inverse tail distribution 
\[
F^{\rm IT}_H(\tau) = \begin{cases}
	1 + \frac{\lambda}{\lambda - \mu}\left(e^{(\lambda - \mu) \tau} - 1\right)	& \lambda \ne \mu; \\
	1 + \lambda \tau								& \lambda = \mu. 
	\end{cases}
\]
The corresponding node height {\em cdf} is 
\begin{equation}	\label{birthDeathFH}
F_H(\tau) = \begin{cases}
	\frac{\lambda(e^{(\lambda - \mu) \tau} - 1)}{\lambda e^{(\lambda - \mu) \tau} - \mu} 	& \lambda \ne \mu; \\
	\frac{\lambda \tau}{1 + \lambda \tau}.						& \lambda = \mu. 
	\end{cases}
\end{equation}
\end{lemma}
\begin{proof}
Start with the distribution of $N(t)$ conditional on a single founder at time 0 \citep[p480]{Feller68}, 
\[
\mathbb{P}(N(t) = n \mid N(0) = 1) = \begin{cases}
	\mu B(t) 									& n= 0 \\
	(1 - \lambda B(t))(1 - \mu B(t))(\lambda B(t))^{n - 1}	& n = 1, 2, \ldots, 
	\end{cases}
\]
where 
\[
B(t) = \frac{e^{(\lambda - \mu) t} - 1}{\lambda e^{(\lambda - \mu) t} - \mu}.  
\]
Conditional on non-extinction of the process up to time $t$, 
\[
\mathbb{P}(N(t) = n \mid N(0) = 1, N(t) > 0) = (1 - \lambda B(t))(\lambda B(t))^{n - 1}, \qquad n = 1, 2, \ldots, 
\]
and so $(N(t) \mid N(0) = 1, N(t) > 0)$, and hence the number of coalescent events in the interval $[0, t]$, is distributed as 
shifted geometric with ``success'' probability $\lambda B(t)$ for any $t \ge 0$.  This can be achieved with a node height distribution 
\[
\mathbb{P}(H \le \tau) = \lambda B(\tau) = \lambda \frac{e^{(\lambda - \mu) \tau} - 1}{\lambda e^{(\lambda - \mu) \tau} - \mu}, \qquad \tau > 0, 
\]
agreeing with Eq.~(\ref{birthDeathFH}).  
The $\lambda = \mu$ case follows by taking the limit $\mu \to \lambda$.  
\end{proof}
 From Eqs.~(\ref{normalisedFH}), (\ref{T2_to_Tn_cdf}), (\ref{T2_to_Tk_density}) and (\ref{birthDeathFH}) one obtains the joint {\em cdf} and density 
 of coalescent times for a BD process conditioned on $n$ leaves and a time of origin $t$, as in \citet[Section~6.1]{Gernhard08}.  

Note that if $\lambda < \mu$, then $\lim_{\tau \to \infty} F_H(\tau) = \lambda/\mu < 1$, indicating that $\mathbb{P}(H = \infty) > 0$, 
or in other words, that the RRP may not converge to a single ancestral founder as $t \to \infty$.  If one is specifically interested in processes 
initiated from a single ancestral founder 
it is convenient to define an alternate node height conditioned on convergence to  single ancestor with rescaled {\em cdf} 
\begin{equation}	\label{FHaltBD}
F^{\rm BD}_H(\tau)  = \frac{F_H(\tau)}{F_H(\infty)} = \max(\lambda, \mu) \frac{e^{(\lambda - \mu) \tau} - 1}{\lambda e^{(\lambda - \mu) \tau} - \mu}, \qquad \tau, \mu, \lambda \ge 0.  
\end{equation} 
Also define the reparametrisation introduced by \citet[Section~3.2]{Wiuf18} and \citet[Section~2]{Crespo21}
\begin{equation}	\label{WiufParams}
\delta = |\lambda - \mu|, \qquad \gamma = \frac{|\lambda - \mu|}{\max(\lambda, \mu)}.  
\end{equation}
The rescaled node height {\em cdf} is then 
\begin{equation}	\label{FHdeltagamma}
F^{\rm BD}_H(\tau)  = \frac{e^{\delta\tau} - 1}{e^{\delta\tau} - 1 + \gamma}, \qquad \tau, \gamma, \delta > 0,   
\end{equation}
and node height density is 
\[
f^{\rm BD}_H(\tau)  = \frac{\gamma\delta e^{\delta t}}{(e^{\delta t} - 1 + \gamma)^2}, \qquad \tau, \gamma, \delta > 0.  
\]
An important observation is that the node height {\em cdf}, and hence the coalescent tree of a BD process conditioned on convergence to a 
single ancestral founder, is invariant under interchange of $\lambda$ and $\mu$, as observed previously by \citet{Stadler19}.  

%

\section{Feller diffusion as the limit of a CPP}
\label{sec:FellerCPP}

The Feller diffusion arises as the limit of a BD process in which the birth and death rates become infinite, but their difference remains finite.  
In this section, the coalescent tree of a Feller diffusion conditioned on a single ancestral founder is constructed from the limit of a CPP corresponding to 
a BD process with large birth and death rates.  

Consider a linear BD process $\big(N_\epsilon(t)\big)_{t \in \mathbb{R}_{\ge 0}}$ with birth rate $\lambda_\epsilon$ and death rate $\mu_\epsilon$ 
specified as functions of a parameter 
$\epsilon \in \mathbb{R}_{> 0}$ with the property 
\[
\begin{split}
\lambda_\epsilon &= \tfrac{1}{2} \epsilon^{-1} + \tfrac{1}{2} \alpha + \mathcal{O}(\epsilon), \\
 \mu_\epsilon        &= \tfrac{1}{2}\epsilon^{-1} - \tfrac{1}{2} \alpha + \mathcal{O}(\epsilon), 
\end{split}
\]
as $\epsilon \to 0$ for fixed $\alpha \in \mathbb{R}$.  
If $N_\epsilon(t)$ is the number of particles alive at time $t$, set $X(t) = \epsilon N_\epsilon(t)$.  
Then the limiting generator of the process $\big(X(t)\big)_{t \in \mathbb{R}_{\ge 0}}$ as $\epsilon \to 0$ is the Feller diffusion generator 
Eq.~(\ref{FellerGenerator}).  For a proof of this, see \citet[Section~2.1]{BurdenGriffiths24}.  

Specifically, for a given $\alpha \in \mathbb{R}\backslash\{0\}$ choose the 1-parameter family of BD processes with birth and death rates 
\begin{equation}	\label{BDlimit}
\lambda_\epsilon = \tfrac{1}{2} \epsilon^{-1} + \tfrac{1}{2} \alpha, \quad \mu_\epsilon = \tfrac{1}{2}\epsilon^{-1} - \tfrac{1}{2} \alpha, \qquad
		\epsilon > 0. 
\end{equation}
In terms of the parametrisation of a BD process defined by Eq.~(\ref{WiufParams}) this set is equivalently parameterised by   
\begin{equation}	\label{deltaGamma_epsilon}
\delta = |\alpha|, \quad \gamma_\epsilon = \frac{2|\alpha|}{(\epsilon^{-1} + |\alpha|)} = 2|\alpha|\epsilon + \mathcal{O}(\epsilon^2) \quad \text{as }\epsilon \to 0, 
\end{equation}
which inverts to 
\[
\alpha = \pm\delta, \quad \epsilon = \frac{\gamma_\epsilon}{2\delta} + \mathcal{O}(\gamma_\epsilon^2) \quad \text{as }\gamma_\epsilon \to 0. 
\]
Clearly $\delta$ is independent of $\epsilon$ and the limit $\gamma_\epsilon \to 0$ implies $\epsilon \to 0$.  
Suppressing the explicit $\epsilon$ dependence, it follows that the limit $\gamma \to 0$ at fixed $\delta$ of any property of a coalescent tree 
calculated using the node height density Eq.~(\ref{FHdeltagamma}) is that of a Feller diffusion with parameter $\alpha = \pm \delta$,  
where the upper sign corresponds to a super-critical Feller diffusion and the lower sign corresponds to a sub-critical Feller diffusion.  
An important consequence is that coalescent tree of a Feller diffusion conditioned on a given observed population converging to a single 
ancestral founder is invariant under a change in sign of $\alpha$ (see Theorem~\ref{thm:JointTkGivenX} below). 
In general, one finds that properties for a critical Feller diffusion can be obtained by taking the limit $\alpha \to 0$ at the end of the day.  

Heuristically, the $\gamma \to 0$ limit corresponds to a node height density $f_H^{\rm BD}(\tau)$ which becomes heavily 
skewed towards zero, leading to an infinite number of leaves and a high density of nodes close to the base in the coalescent tree in 
Fig.~\ref{fig:CPP_diagram}.  In other words, the coalescent tree `comes down from infinity'.  This behaviour has been noted and recognised 
as a Feller diffusion limit by \citet[Section~3.2, Fig.~2 and Section~5]{Aldous05} for the critical BD process.  

With this machinery set up we have an alternative proof of the following theorem, which is seminal to a number of results in 
\citet{BurdenGriffiths25}\footnote{See also \citet[Thm.~4]{BurdenGriffiths24} for a proof similar to that given in \ref{sec:JointTkProof}, but specifically 
for the case of a super-critical diffusion with an improper uniform prior on $T_1$.}.  
\begin{theorem} \citep[Thm.~1]{BurdenGriffiths25}	\label{thm:JointTkGivenX}
Consider a Feller diffusion $\big(X(u)\big)_{u \in [0, t]}$ with parameter $\alpha \in \mathbb{R}$, 
descended from a single founder at time zero and conditioned on a current population $X(t) = x$. 
The joint density of the $k - 1$ population coalescent times $T_2, \ldots, T_k$ is 
\begin{eqnarray}	\label{jointTkGivenx}
f_{T_2, \ldots, T_k \mid T_1 = t, X(t) = x}(\tau_2, \ldots, \tau_k) & = & \left(\prod_{j = 2}^k \frac{x \mu(\tau_j; |\alpha|)}{2\beta(\tau_j; |\alpha|)}\right) 
									e^{-(x/\beta(\tau_k; |\alpha|) - x/\beta(t; |\alpha|))}, \nonumber \\
	&  & \qquad\qquad\qquad t > \tau_2 > \tau_3 > \ldots > \tau_k > 0, 
\end{eqnarray}
and the marginal density of of the population coalescent time $T_k$ is 
\begin{eqnarray}	\label{TkGivenx}
f_{T_k \mid T_1 = t, X(t) = x}(\tau) & = & \frac{(x/\beta(\tau; |\alpha|) - x/\beta(t; |\alpha|))^{k - 2}}{(k - 2)!} 
							\frac{x \mu(\tau; |\alpha|)}{2\beta(\tau; |\alpha|)} 
									e^{-(x/\beta(\tau_k; |\alpha|) - x/\beta(t; |\alpha|))}, \nonumber \\
	& & \qquad\qquad\qquad\qquad\qquad  k = 2, 3, \ldots, \quad 0 <  \tau < t,  
\end{eqnarray}
where 
\begin{equation}	\label{betaDef}
\beta(\tau; \alpha) := \frac{e^{\alpha\tau} - 1}{2\alpha},  \quad \mu(\tau; \alpha) := \frac{2\alpha e^{\alpha \tau}}{e^{\alpha\tau} - 1} \qquad \alpha \ne 0, 
\end{equation}
and $\beta(\tau; 0) := \tau/2$, $\mu(\tau; 0) := 2/\tau$.  
\end{theorem}
The proof of Theorem~\ref{thm:JointTkGivenX} is given in \ref{sec:JointTkProof}.  
In addition we have
\begin{theorem} 
Consider a Feller diffusion $\big(X(u)\big)_{u \in [0, t]}$ with parameter $\alpha \in \mathbb{R}$, 
descended from a single founder at time zero. Independent of the final population $X(t)$, the density of the coalescent time $T_i$ conditional 
on $T_j$ for $i < j$ is
\begin{eqnarray}	\label{Ti_given_TjFeller}
f_{T_i \mid T_1 = t, T_j = s}(\tau) & = & \frac{(j - 2)!}{2(i - 2)!(j - i - 1)!} \times\nonumber \\
	& &  \quad \mu(\tau) \frac{\beta(s)^{i - 1} \beta(t)^{j - i}}{\beta(\tau)^{j - 2}} 
						\frac{(\beta(\tau) - \beta(s))^{j - i - 1}(\beta(t) - \beta(\tau))^{i - 2}}
								{(\beta(t) - \beta(s))^{j - 2}}, \nonumber \\
	& & \qquad\qquad\qquad\qquad  2 \le i < j \le n, \quad t > \tau > s > 0.  
\end{eqnarray}
where we have written $\beta(\cdot)$ for $\beta(\cdot\, ; |\alpha|)$ and $\mu(\cdot)$ for $\mu(\cdot\, ; |\alpha|)$.
\end{theorem}
\begin{proof}
The result follows from Eq.(\ref{Ti_given_Tj}) using the same limiting behaviours as in the previous theorem.
\end{proof}

\begin{remark}	\label{RemarkCrespo}
\citet[Section~3]{Crespo21} refer to the above limiting process as the Limit Birth-Death process, but appear to be unaware that it 
corresponds to a Feller diffusion.  To convert from the $t \to \infty$ limit of Eq.~(\ref{TkGivenx}) to the notation of \citet[Eq.~5]{Crespo21}, 
replace $2|\alpha| x$ with $\Gamma$, $|\alpha| \tau$ with $\Delta u$, and $k$ with $l + 1$ to allow for the fact that \citet{Crespo21} have 
assumed an improper uniform$[0, \infty]$ prior on $T_1$ (see Eq.~(\ref{fgivenT1tofunif})).  To convert from the $t \to \infty$ limit of 
Eq.~(\ref{Ti_given_TjFeller}) to \citet[Eq.~4]{Crespo21}, replace $|\alpha| s$ with $\Delta u$, $|\alpha| \tau$ with $\Delta (u + v)$, 
$i$ with $l + 1$ and $j$ with $k + 2$.  
\end{remark}
%
%

\section{Sampling}
\label{sec:Sampling}

The following quote from \citet[Section~1.1, final paragraph]{Lambert18} summarises two sampling schemes of interest.  
\begin{quote}
In the biology literature, there are mainly two classical sampling schemes [\ldots].  The first scheme, called the {\em Bernoulli sampling} scheme, 
consists of selecting each extant tip with the same probability.  The second scheme, called {\em $k$-sampling} scheme, consists of drawing 
uniformly $k$ tips among the extant tips of the splitting tree conditioned upon [the number of extant tips] $> k$.  
\end{quote}
Bernoulli sampling is a method employed in survey statistics, database analysis and modelling binary outcomes, popular for its ease of 
implementation and scalability, but with the disadvantage that the sample size is necessarily random.  Our interest in the current 
paper however lies in exploiting a mathematical connection between Bernoulli sampling and CPPs.

More specifically, \citet[Section~3.3]{Lambert13b} have shown that a Bernoulli sampled CPP with sampling rate $\rho \in [0, 1]$ is itself a CPP with typical node height
$H_\rho$, say, with inverse tail distribution and {\em cdf} 
\begin{equation}	\label{FHRho}
\begin{split}
F^{\rm IT}_{H_\rho}(\tau) 	&= 1 - \rho + \rho F^{\rm IT}_H(\tau), \\
F_{H_\rho}(\tau) 		&= \frac{\rho F_H(t)}{1 - (1 - \rho)F_H(t)}.  
\end{split}
\end{equation}
On the other hand, a $k$-sampled CPP is not a CPP.  The full story of $k$-sampling is more complicated \citep[Sections~3 and 4]{Lambert18}.  
We explore $k$-sampling for a Feller diffusion in Subsections~\ref{sec:kSamplingFeller}, \ref{sec:kSampledT2toTk} and \ref{sec:kSampledT2toTkXtConditioned}.  

%

\subsection{Bernoulli sampling for a BD process}
\label{sec:BSamplingBD}

Coalescent times for a conditioned Bernoulli-sampled linear BD process were first calculated directly from Bayes' law by \citet{Stadler09}.  
However, use of CPPs makes the task considerably easier.  

The typical node height {\em cdf} for a Bernoulli sampled BD process with constant rates $\lambda$ and $\mu$ and sampling rate $\rho$ is, 
from Eqs.(\ref{birthDeathFH}) and (\ref{FHRho}),
\[
F_{H_\rho}(\tau) = \begin{cases}
	\frac{\rho\lambda(e^{(\lambda - \mu)\tau} - 1)}{\rho\lambda e^{(\lambda - \mu)\tau} - \mu + \lambda(1 - \rho)} 	& \lambda \ne \mu; \\
	\frac{\rho\lambda\tau}{1 + \rho\lambda\tau}.			& \lambda = \mu, 
	\end{cases} \quad\tau \ge 0. 
\]
As $\tau \to \infty$ we have 
\[
F_{H_\rho}(\infty) = \begin{cases}
	1	& \lambda \ge \mu; \\
	\frac{\rho\lambda}{\mu - \lambda(1 - \rho)}.			& \lambda < \mu. 
	\end{cases}    
\]
Following the logic leading to Eq.~(\ref{FHaltBD}), to condition on convergence of the RRP to a single ancestor in the sub-critical case, 
define an alternate node height {\em cdf} for $\lambda \ne \mu$ by 
\begin{eqnarray*}	
F_{H_\rho}^{\rm BD}(\tau) & := & \frac{F_{H_\rho}(\tau)}{F_{H_\rho}(\infty)} \nonumber \\ 
	& = & \begin{cases} 
			\rho\lambda\frac{e^{(\lambda - \mu)\tau} - 1}{\rho\lambda e^{(\lambda - \mu)\tau} -[\mu - \lambda(1 - \rho)]} 	& \lambda > \mu; \\
			[\mu - \lambda(1 - \rho)]\frac{e^{(\lambda - \mu)\tau} - 1}{\rho\lambda e^{(\lambda - \mu)\tau} - [\mu - \lambda(1 - \rho)]} 	& \lambda < \mu; 
		\end{cases} \nonumber \\
	& = & \begin{cases} 
			 \frac{e^{(\lambda - \mu)\tau} - 1}{e^{(\lambda - \mu)\tau} - 1 + \frac{\lambda - \mu}{\rho\lambda}} 	& \lambda > \mu; \\
			 \frac{e^{(\mu - \lambda)\tau} - 1}{e^{(\mu - \lambda)\tau} - 1 + \frac{\mu - \lambda}{\mu - (1 - \rho)\lambda}} 	& \lambda < \mu. 
		\end{cases}
\end{eqnarray*} 
The two cases can be combined into a single form 
\begin{equation}	\label{FHrhoAltBD}
F_{H_\rho}^{\rm BD}(\tau) = \frac{e^{\delta\tau} - 1}{e^{\delta\tau} - 1 + \gamma_\rho}, \qquad \tau, \delta, \gamma_\rho > 0, 
\end{equation}
where $\delta$ is defined by Eq.~(\ref{WiufParams}) and 
\begin{equation}	\label{WiufParamsRho}
\gamma_\rho = \frac{|\lambda - \mu|}{\max(\rho\lambda, \mu - (1 - \rho)\lambda)}.  
\end{equation}
This form of the node height {\em cdf} was found previously by \citet[Eq.~(13)]{Wiuf18} by working from the distribution of sampled 
particles rather than the associated CPP.  It takes the same form as Eq.~(\ref{FHdeltagamma}), and hence leads to an identical 
set of coalescent time distributions up to a transformation of parameters.  \citet[Sections~3.3 and 6.1]{Wiuf18} makes the point that, of the three 
parameters $(\mu, \nu\ \rho)$, only two degrees of freedom are identifiable, and identifies symmetries resulting from reversal of birth and 
death rates.  The same symmetries have also been identified by \citet{Stadler19}, and, for the supercritical case, by \citet[Remark~3.4]{Stadler09}.  

%

\subsection{Poisson sampling for a Feller diffusion}
\label{sec:PSamplingFeller}

To approach the Feller diffusion limit of a Bernoulli-sampled BD process in a way consistent with the scaling limit $\epsilon \to 0$ defined in 
Eq.~(\ref{BDlimit}), define an $\epsilon$-dependent Bernoulli sampling rate 
\[
\rho_\epsilon := 2\nu\epsilon, \quad \nu > 0, 
\]
and take the limit $\epsilon, \rho_\epsilon \to 0$ while keeping the newly introduced parameter $\nu$ fixed.  
We will refer to the limiting sampling scheme\footnote{Not to be confused with the use of the term ``Poisson sampling'' in the survey sampling literature, 
which refers to a generalised form of Bernoulli sampling in which the sampling rate varies with each sampled individual~\citep[Section~3]{Tille17}.} 
as {\em Poisson sampling with parameter $\nu$}, because the scaling is chosen so that, 
in the diffusion limit $X(t) = \lim_{\epsilon \to 0} \epsilon N_\epsilon(t)$, the sample size conditioned on a population size $X(t) = x$ 
is distributed asymptotically as 
\[
{\text{Binom}(x/\epsilon, 2\epsilon\nu) \sim \text{Pois}(2\nu x)}. 
\]
Physically, this means that, if the observed scaled population is $x \in \mathbb{R}_{\ge 0}$, randomly sample $N$ individuals from the population, 
where $N \sim \text{Pois}(2\nu x)$.  The scheme is intended primarily as a mathematical expedient for approaching the Feller diffusion.  

In the diffusion limit of the Bernoulli-sampled BD process, $\delta = |\lambda_\epsilon - \mu_\epsilon| = |\alpha|$ is $\epsilon$-independent as before, 
and the parameter $\gamma_\rho$ has a finite limit.  From Eqs.~(\ref{BDlimit}) and (\ref{WiufParamsRho}) 
\begin{eqnarray}	\label{gammaNuDef} 
\gamma_\nu & := & \lim_{\epsilon \to 0} \gamma_\rho \nonumber \\
	& = & \lim_{\epsilon \to 0} \frac{|\lambda_\epsilon - \mu_\epsilon|}
				{\max(\rho_\epsilon\lambda_\epsilon, \mu_\epsilon - (1 - \rho_\epsilon)\lambda_\epsilon)} \nonumber \\
	& = & \frac{|\alpha|}{\max(\nu, \nu - \alpha)} \nonumber \\
	& = & \begin{cases}
			\frac{\alpha}{\nu}, & \alpha > 0; \\
			\frac{\alpha}{\alpha - \nu}, & \alpha < 0. \\	
			\end{cases} 
\end{eqnarray} 
The coalescent tree of a Poisson-sampled Feller diffusion can therefore be constructed from a CPP with node height {\em cdf} 
Eq.~(\ref{FHrhoAltBD}), but with $\gamma_\rho$ replaced by $\gamma_\nu$, that is 
\begin{eqnarray}	
F_{H_\nu}^{\rm Feller}(\tau) & = & \frac{e^{\delta\tau} - 1}{e^{\delta\tau} - 1 + \gamma_\nu}, \qquad \tau, \delta, \gamma_\nu > 0  	\label{FHnuFeller} \\
	& = & \begin{cases}
			\displaystyle\frac{2\beta(\tau; \alpha)}{2\beta(\tau; \alpha) + 1/\nu}	& \alpha > 0;	\\ \\
			\displaystyle\frac{\nu - \alpha}{\nu}\frac{2\beta(\tau; \alpha)}{2\beta(\tau; \alpha) + 1/\nu}	& \alpha < 0; \\ \\
			\displaystyle\frac{\tau}{\tau + 1/\nu} & \alpha = 0, 
			\end{cases}		\label{FHnuAltFeller}			
\end{eqnarray}
where $\beta(\tau; \alpha)$ is defined by Eq.~(\ref{betaDef}).  
  
From Eq.~(\ref{gammaNuDef}) it is obvious that $\gamma_\nu$ and $\delta$ remain invariant under the transformation 
\[
(\alpha, \nu) \rightarrow (\alpha', \nu') = (-\alpha, \nu - \alpha),  
\]
and that, provided $\nu - \alpha > 0$, this defines a new Poisson-sampled Feller process with the same distribution of 
coalescent times.  This can be summarised as: 
\begin{corollary}
Given a Poisson-sampled super-critical Feller diffusion with parameters $\nu > \alpha > 0$, there exists a Poisson-sampled 
sub-critical Feller diffusion with Feller parameter $-\alpha < 0$, and sampling rate $\nu - \alpha > 0$ with the same distribution of coalescent times.  

Given a Poisson-sampled super-critical Feller diffusion with parameters $\alpha > \nu > 0$, there does not exist a Poisson-sampled 
sub-critical Feller diffusion with the same distribution of coalescent times.  

Given a Poisson-sampled sub-critical Feller diffusion with parameters $\alpha < 0$, $\nu > 0$, there exists a Poisson-sampled 
super-critical Feller diffusion with Feller parameter $-\alpha > 0$, and sampling rate $\nu - \alpha > 0$ with the same distribution of coalescent times.  
\end{corollary}
These symmetries are analogues of the symmetries set out in \citet[Theorem~8]{Stadler19} for a Bernoulli-sampled BD process.  

%

\subsection{$k$-sampling for a Feller diffusion}
\label{sec:kSamplingFeller}

\citet{Lambert18} gives a procedure for determining the coalescent properties of a $k$-sampled population resulting from a CPP 
in terms of the corresponding properties of a Bernoulli-sampled population
conditioned on a sample size $k$.  Lambert's main result is the following theorem.  
\begin{theorem} \citep[Thm.~3]{Lambert18}	\label{thm:Lambert}
Consider a CPP with node height $H$ stopped at time $t$.  Fix $k \ge 1$.  Let $\Pi_k$ denote the law of the tree generated by $k$ tips 
sampled uniformly from the tips of the CPP (conditioned to have at least $k$ tips), and for each $\rho \in (0, 1)$ let $\Gamma_{\rho, k}$ denote 
the law of the Bernoulli-sampled CPP with sampling probability $\rho$, conditioned to have $k$ tips.  Then 
\[
\Pi_k = \int_0^1 \Gamma_{\rho, k}\, \mu_k(d\rho), 
\]
where $\mu_k$ is the probability distribution defined by 
\begin{equation}	\label{muKDRhoDef}
\mu_k(d\rho) := \frac{k(1 - a)\rho^{k - 1}}{(1 - a(1 - \rho))^{k + 1}} d\rho, \quad \rho \in (0, 1), 
\end{equation}
and $a = \mathbb{P}(H \le t)$.  
\end{theorem}
\begin{remark}	\label{Remark:BernSampl}
Note that the two laws,  $\Pi_k$ and $\Gamma_{\rho, k}$, differ because of the random variability in the number of tips of the full unsampled tree.  If the full 
tree were further conditioned to have a given number $n \ge k$ tips, say, then Bernoulli sampling with rate $\rho = k/n$ conditioned on $k$ sampled tips 
would be precisely $k$-sampling from a population of size $n$.  
\end{remark}
The following theorem gives the analogous measure to Eq.~(\ref{muKDRhoDef}) for calculating coalescent properties of a $k$-sampled Feller diffusion 
in terms of the of the coalescent properties of a Poisson-sampled a Feller diffusion.  
\begin{theorem}	\label{thm:BernoulliToPoisson}
Consider a Feller diffusion $\big(X(u)\big)_{u \in [0, t]}$ with parameter $\alpha \in \mathbb{R}$ descended from a single ancestor at time zero and 
stopped at time $t$.  Let $\Pi^{(\alpha)}_k$ denote the law of the coalescent tree of $k$ individuals sampled uniformly from the population $X(t)$, 
and for each $\nu > 0$ let $\Gamma^{(\alpha)}_{\nu, k}$ denote the law of the coalescent tree of the Poisson-sampled population $X(t)$ with sampling 
parameter $\nu$, conditioned to have $k$ tips.  Then 
\[
\Pi^{(\alpha)}_k = \int_0^\infty \Gamma^{(\alpha)}_{\nu, k}\, \mu_k(d\nu), 
\]
where 
\begin{equation}		\label{mukOfDnu}
\mu_k(d\nu) =  \frac{k v^{k - 1}}{(1 + v)^{k + 1}}  dv
		, \quad v \in (0, \infty),  
\end{equation}
and
\[
v = 2\nu\beta(t; |\alpha|).  
\]
\end{theorem}
The proof of Theorem~\ref{thm:BernoulliToPoisson} is given in \ref{sec:proofBToP}.  

%

\subsection{Distribution of coalescent times for a $k$-sampled Feller diffusion without conditioning on $X(t) = x$}
\label{sec:kSampledT2toTk}

As an example of the use of Eq.~(\ref{mukOfDnu}) we calculate the joint distribution of coalescent times for a $k$-sampled Feller diffusion 
$\big(X(\tau)\big)_{\tau \in [0, t]}$ initiated from a single ancestor at time 0 and stopped at time $t$.  We do not condition on a specific value
for $X(t)$, the point being that Theorem~\ref{thm:Lambert} requires only that the population size is at least $k$, which will always be the case 
in the diffusion limit if non-extinction is assumed.  This is identical to the problem of coalescent times for a long-term, near critical limit of a BD 
process addressed by \citet[Theorem~3]{Harris20}.  
\begin{theorem}	\label{AltHarrisTheorem}
Consider a Feller diffusion $\big(X(\tau)\big)_{\tau \in [0, t]}$ with parameter $\alpha$ initiated from a single ancestor at time 0 and stopped at time $t$.  The joint 
distribution of coalescent times for a sample of size $k$ drawn uniformly from the final population $X(t)$ is 
\[
F_{T_2\ldots T_k}^{k\text{-}{\rm sample}}(\tau_2, \ldots, \tau_k \mid T_1 = t) 
	= k! \left( C - \sum_{j = 2}^k G_j \log B(\tau_j) \right),   \qquad t > \tau_2 > \ldots > \tau_k > 0, 
\]
where 
\[
B(\tau) = \frac{\beta(t; \alpha)}{\beta(\tau; \alpha)}, 
\]
\[
C = \prod_{j = 2}^k \frac{\beta(\tau_j; \alpha)}{\beta(\tau_j; \alpha) - \beta(t; \alpha)}, 
\]
and 
\[
G_i = - \frac{\beta(\tau_j; \alpha)\beta(t; \alpha)}{(\beta(\tau_j; \alpha) - \beta(t; \alpha))^2} 
							\prod_{l = 2:k, \, l \ne i} \frac{\beta(\tau_l; \alpha)}{\beta(\tau_l; \alpha) - \beta(t; \alpha)}, 
\]
where $\beta(\tau; \alpha)$ is defined by Eq.~(\ref{betaDef}).  
\end{theorem}

The proof, starting from Theorem~\ref{thm:BernoulliToPoisson}, is given in \ref{sec:HarrisProof}.  
This result agrees with \citet[Theorem~3]{Harris20}, up to a factor $(k - 1)!$ to account for Harris et al.'s unordered coalescent 
times (defined at \citet[p1371]{Harris20}) compared 
with our ordered coalescent times.  To convert from our notation to that of Harris et al., replace $e^{\alpha\tau_j} - 1$ with $E_{j - 1}$ for $j = 2, \ldots, k$ and 
$e^{\alpha t} - 1$ with $E_0$ (see \citet[Table~C.1]{BurdenGriffiths25} for details).  To get the critical case use $\beta(\tau; 0) = \tau/2$.  

For $k = 2$, we have $C = -1/(B(\tau_2) - 1)$ and $G_2 = -B(\tau_2)/(1 - B(\tau_2))^2$ by interpreting the product of no factors as 1.  For the time $T_2$ since the MRCA 
of a random sample of 2 individuals in a population we get 
\begin{equation}	\label{T2_2Sample}
F_{T_2}^{2\text{-}{\rm sample}}(\tau \mid T_1 = t) = \frac{2\beta(\tau)}{\beta(\tau) - \beta(t)} 
				+ \frac{2\beta(\tau)\beta(t)}{(\beta(\tau) - \beta(t))^2} \log\frac{\beta(t)}{\beta(\tau)}, 
\end{equation}
which agrees with a result obtained by \citet[Eq.~(14)]{BurdenGriffiths25} from coalescent properties deduced from the solution to the Feller diffusion 
forward Kolmogorov equation.  

%
\subsection{Distribution of coalescent times for a $k$-sampled Feller diffusion with conditioning on $X(t) = x$}
\label{sec:kSampledT2toTkXtConditioned}

Theorem~\ref{thm:BernoulliToPoisson} is not directly applicable when conditioning on a specified final population $X(t) = x$ as it relies, via 
Theorem~\ref{thm:Lambert}, on the random variability of $X(t)$ (see Remark~\ref{Remark:BernSampl}).  An alternative approach to $k$-sampling 
is that due to \cite{Crespo21}, who compute simulated coalescent trees of $k$-sampled populations from a scaled limit of a BD process which they 
refer to as a Limit Birth-Death process.  As pointed out in Remark~\ref{RemarkCrespo} this process is equivalent to a Feller process conditioned 
on a current population $X(t) = x$ descended from a single founder at time zero.  Their method employs a result of \citet{Saunders84} to 
compute the probability distribution of coalescent times for a uniform random sample of size $k$ among the coalescent times of an infinite population.  

Using the conventions in Fig.~\ref{fig:CPP_diagram}, let the coalescent times measured back from the present be 
$0 < \tilde{T}_k < \ldots < \tilde{T}_1$ for the sample, and $ 0 < \ldots < T_2 < T_1$ for the infinite population from which the sample is taken.  
Clearly $\{\tilde{T}_l : 1 \le l \le k\} \subset \{T_i: 1 \le i < \infty\}$.  
Then for any model leading to a binary exchangeable tree, 
\begin{equation}	\label{PTtildenEqualsT}
\mathbb{P}(\tilde{T}_k = T_i) = k(k - 1) \frac{(i - 1)!(i - 2)!}{(i - k)!(i + k - 1)!}, \qquad i = k, k + 1,\ldots, 
\end{equation}
and 
\begin{eqnarray}	\label{PTtildekEqualsT}
\lefteqn{\mathbb{P}(\tilde{T}_l = T_i \mid \tilde{T}_{l + 1} = T_j) } \nonumber \\ 
	& = & l(l - 1) \frac{(j - l - 1)!(j + l - 2)!}{(j - 1)!(j - 2)!} \frac{(i - 1)!(i - 2)!}{(i - l)!(i  + l - 1)!}, \nonumber \\
	& & \qquad\qquad\qquad\qquad i = l, \ldots, j - 1; \quad l = 2, \ldots, k - 1.  
\end{eqnarray}
See \ref{sec:matchingTtildeT} for details.  

In principle, $k$-sampled coalescent trees for a Feller diffusion $\big(X(u)\big)_{u \in [0, t]}$ conditioned on a current population $X(t) = x$ 
and descended from a single founder can be simulated by first drawing a set of population coalescent time indices $\{i_l\}$ corresponding to the 
sample coalescent times $\tilde{T}_l$ from the distribution defined by Eqs.~(\ref{PTtildenEqualsT}) and (\ref{PTtildekEqualsT}), and then 
drawing coalescent times $T_{i_l}$ from the densities Eqs.~(\ref{TkGivenx}) and (\ref{Ti_given_TjFeller}).  \cite{Crespo21} have  
simulated trees corresponding to samples sample of size $k = 10$ and $100$, a range of values of $0.001 \le 2\alpha x \le 1000$, and assuming 
an improper uniform prior on $T_1$ (see Subsection~\ref{sec:T1unifPrior}).  They report that the main determinant of computational speed is 
the generation of coalescent time indices $i_l$.  They also explore the computational speed of various semi-deterministic approximations with 
varying success.  

The method is not restricted to diffusions conditioned on $X(t) = x$ but can be applied more generally.  For example, for the Feller diffusion 
$\big(X(u)\big)_{u \in [0, t]}$ conditioned on a non-extinct current population $X(t) > 0$ descended from a single founder, but otherwise unrestricted, 
\citet[Thm.~2]{BurdenGriffiths25} obtain for the density of the population coalescent time $T_i$ 
\[
f_{T_i}(\tau \mid T_1 = t) = \tfrac{1}{2}(i - 1)\frac{\mu(\tau)\beta(\tau)}{\beta(t)} \left(1 - \frac{\beta(\tau)}{\beta(t)}\right)^{i - 2}, 
			\qquad 0 < \tau < t, \quad i = 2, 3, \ldots,
\]
with $\mu(\tau)$ and $\beta(\tau)$ defined by Eq.~(\ref{betaDef}).  The {\em cdf} for $T_i$ is therefore 
\[
F_{T_i}(\tau \mid T_1 = t) = 1 - \left(1 - \frac{\beta(\tau)}{\beta(t)}\right)^{i - 1}, \qquad 0 < \tau < t, \quad i = 2, 3, \ldots
\]
The {\em cdf} for coalescent time $\tilde{T}_2$ corresponding to the MRCA  of a random sample of size 2 is then 
\begin{eqnarray*}
F_{\tilde{T}_2}^{2\text{-}{\rm sample}}(\tau \mid T_1 = t) 
	& = & \sum_{i = 2}^\infty F_{T_i}(\tau \mid T_1 = t) \times \mathbb{P}(\tilde{T}_2 = T_i) \nonumber \\
	& = & \sum_{i = 2}^\infty \left\{ 1 - \left(1 - \frac{\beta(\tau)}{\beta(t)}\right)^{i - 1} \right\} \times \frac{2}{i(i - 1)} \nonumber \\
	& = & 1 - 2 \sum_{i = 2}^\infty \left(\frac{1}{i} - \frac{1}{i - 1} \right) \left(1 - \frac{\beta(\tau)}{\beta(t)}\right)^{i - 1}, \qquad 0 < \tau < t,  
\end{eqnarray*}
from which it is straightforward to reproduce Eq.~(\ref{T2_2Sample}) with help from the identity $\sum_{j = 1}^\infty (1 - z)^j/j = \log z^{-1}$.  

%

\section{Expected inter-coalescent waiting times for trees constructed from Wiuf's node-height distribution}
\label{sec:ExpectedWkWiuf}

\citet[Section~8.2 and Appendix~G]{Wiuf18} gives an explicit formula for the inter-coalescent waiting time $W_1 = T_1 - T_2$ subject to an 
improper uniform$[0,\infty)$ prior on the time $t$ since initiation from a single ancestor, and conditioned on $n$ leaves for processes whose 
coalescent tree corresponds to a CPP with node-height {\em cdf} taking the generic form 
\begin{equation}	\label{FHWiuf}
F_H^{\rm Wiuf}(\tau) := \frac{e^{\delta\tau} - 1}{e^{\delta\tau} - 1 + \gamma}, \qquad \tau, \delta, \gamma > 0. 
\end{equation}
This includes the BD process (Eq.(\ref{FHdeltagamma})), the Bernoulli-sampled BD process (Eq.~(\ref{FHrhoAltBD})), 
and the Poisson-sampled Feller process (Eq.~(\ref{FHnuFeller})).  Wiuf's result is (see \ref{sec:EnW1Calc}) 
\begin{equation}	\label{EnOfW1}
\mathbb{E}_n^{{\rm unif}\, [0, \infty)}[W_1] = -\frac{n}{\delta} \sum_{i = 1}^{n - 1} \frac{\gamma}{i} (1 - \gamma)^{i - n} -
					\frac{n}{\delta} \frac{\gamma}{(1 - \gamma)^n} \log\gamma, \qquad \gamma \ne 1.  
\end{equation} 
The expected waiting times $W_k$ for $k = 2, \ldots, n$ can thus be computed via the iterative formula Eq.~(\ref{iterativeWkunifT1}).   
The following theorem gives an explicit closed form solution.  

%
\begin{theorem}	\label{theorem:ExpectedWk}
The solution to the recursion Eq.~(\ref{iterativeWkunifT1}) 
with boundary condition Eq.~(\ref{EnOfW1}) for the expectation of the waiting times $W_k$, is 
\begin{equation}	\label{EnOfWkUnifT1}
\mathbb{E}_n^{{\rm unif}\, [0, \infty)}[W_k] = \frac{\gamma}{\delta} \frac{1}{k} \, _2F_1(n - k + 1, 1; n + 1; 1 - \gamma), \qquad \delta, \gamma > 0, \quad 1 \le k \le n, 
\end{equation}
where \citep[Eq.~(15.1.1)]{Abramowitz:1965sf} 
\begin{equation}	\label{HypergeomDef}
_2F_1(a, b; c; z) = 
	\frac{\Gamma(c)}{\Gamma(a) \Gamma(b)} \sum_{l = 0}^\infty \frac{\Gamma(a + l) \Gamma(b + l)}{\Gamma(c + l)} \frac{z^l}{l!}, \qquad {\mathcal R}(c - a - b) > 0, 
\end{equation}
and its analytic continuation, is the ordinary hypergeometric function. 
\end{theorem}
The proof of Theorem~\ref{theorem:ExpectedWk} is given in \ref{sec:proofExpectedWk}.  In numerical calculations, Eq.~(\ref{EnOfWkUnifT1}) 
(or Eq.~(\ref{EnOfWkInfSum}) in the appendix) may be more accurate than iterating from Eq.(\ref{EnOfW1}), 
which, for large $n$ and $\gamma$ close to 1, becomes a small difference of two large numbers.  For certain values of $\gamma$ we have the following 
theorem: 
%
\begin{theorem}	\label{theorem:CasesWk}
The following particular cases apply for the expected inter-coalescent waiting times $W_k$ of the coalescent tree associated with a CPP with node 
height {\em cdf} Eq.~(\ref{FHWiuf}), subject to an improper uniform$[0, \infty)$ prior on $T_1$,  
\begin{equation}	\label{EofWgamma0}
\mathbb{E}_n^{{\rm unif}\, [0, \infty)}[W_k] \sim \begin{cases}
	- \frac{n\gamma}{\delta} \log\gamma, 	\quad k = 1; \\
	\frac{n\gamma}{\delta} \frac{1}{k(k - 1)}, 	\quad k = 2, \ldots, n, 
	\end{cases} 	 \quad\text{as } \gamma \to 0; 
\end{equation}
\begin{equation}	\label{EofWgamma1}
\left. \mathbb{E}_n^{{\rm unif}\, [0, \infty)}[W_k]\right|_{\gamma = 1} = \frac{1}{\delta k}, \quad k = 1, \ldots, n;  
\end{equation}
and 
\begin{equation}	\label{EofWgammaInf}
\mathbb{E}_n^{{\rm unif}\, [0, \infty)}[W_k] \sim \begin{cases}
	\frac{1}{\delta} \frac{n}{k(n - k)}, & k = 1, \ldots, n - 1; \\
	\frac{1}{\delta} \left( \log\gamma - \sum_{i = 1}^{n - 1} \frac{1}{i} \right), & k = n, 
	\end{cases}
		\qquad \text{as }\gamma\to\infty. 
\end{equation}
\end{theorem}
The proof of Theorem~\ref{theorem:CasesWk} is given in \ref{sec:proofCasesWk}.  Note that, up to a factor, the expected coalescent times for 
$\gamma \to 0$ match those of the Kingman coalescent for $k \ge 2$.  From Eqs.~(\ref{WiufParamsRho}) and (\ref{gammaNuDef}) we see that 
$\gamma \to \infty$ corresponds to small sampling rates in the supercritical case, specifically $\rho << (\lambda - \mu)/\lambda$ for a BD process or 
$\nu << \alpha$ for a Feller diffusion.  For an unsampled BD process and for sampled subcritical cases, $\gamma$, $\gamma_\rho$, and 
$\gamma_\nu \in [0, 1]$, and the limit is irrelevant.   

The above results assume an improper uniform prior on the time since initiation of the process.  For completeness, the following corollary to 
Theorem~\ref{theorem:ExpectedWk} covers expected waiting times for a process initiated from a single founding ancestor infinitely far in the past.    
%
\begin{corollary}	\label{cor:ExpectedTk} 
The expected waiting times for a coalescent tree corresponding to a node height {\em cdf} $F_H^{\rm Wiuf}(\tau)$ 
conditioned on $T_1 = \infty$ with $n$ leaves are 
\[
\mathbb{E}_n[W_k \mid T_1 = \infty] = \frac{\gamma}{\delta} \frac{1}{k - 1} \, _2F_1(n - k + 1, 1; n; 1 - \gamma), \qquad \delta, \gamma \ge 0, \quad 2 \le k \le n. 
\]
\end{corollary}
\begin{proof}
From Eq.~(\ref{EgivenT1toEunif}), 
\[
\mathbb{E}_n[W_k \mid T_1 = \infty] = \mathbb{E}_{n - 1}^{{\rm unif}\, [0, \infty)}[W_{k - 1}], \qquad 2 \le k \le n, 
\]
and the result follows from  Eq.~(\ref{EnOfWkUnifT1}).  
\end{proof}

%

\subsection{Comparison with \citet{Ignatieva20}}
\label{sec:ComparisonIgnatieva}

An alternative approach which enables calculation of densities of $W_k$ and simulation of coalescent trees associated with CPPs with 
Wiuf's node-height distribution is that of \citet{Ignatieva20}, 
who construct the RRP as an inhomogeneous pure death process with rate $\mu_{\rm eff}(\tau)$ determined from Eq.~(\ref{FHToMuEff}).  
\citet{Ignatieva20} only claim results for a super-critical Bernoulli-sampled BD process and an improper uniform prior on the time since initiation 
of the process.  However, since their results implicitly rely solely on 
an underlying CPP with node height {\em cdf} of the form of Eq.~(\ref{FHWiuf}), they readily carry over to both Bernoulli-sampled BD processes 
and Poisson-sampled Feller diffusions descended from a single founding ancestor, both in the super-critical and sub-critical cases.   
Furthermore Eq.~(\ref{EgivenT1toEunif}) enables results to carry over to the case of a single founder infinitely far in the past.  

To convert from the notation of \citet[Section~4]{Ignatieva20} to the notation of the current paper, replace $T_k$ and $W_k$ with 
$\delta T_{n - k + 1}$ and $\delta W_{n - k}$ respectively, and replace $\psi\lambda'$ with $\gamma_\rho^{-1}$ in particular for the super-critical 
Bernoulli-sampled BD process.  More generally, we will drop the $\rho$ subscript from here on, keeping in mind that the parameter 
$\gamma$ in Eq.~(\ref{FHWiuf}) can stand for the parameter in any of Eqs.~(\ref{FHdeltagamma}), (\ref{FHrhoAltBD}) or (\ref{FHnuFeller}) 
defined by Eqs.~(\ref{WiufParams}), (\ref{WiufParamsRho}) or (\ref{gammaNuDef}) respectively.  

The inhomogeneous death rate corresponding to process associated 
with the node height distribution Eq.~(\ref{FHWiuf}) is $\mu_{\rm eff}(\tau) = \delta e^{\delta\tau}(e^{\delta\tau} -1 + \gamma)^{-1}$.  
Consistent with Eq.~(\ref{EofWgamma1}), this reduces to a rate-$\delta$ pure-death Yule process with expected waiting times $1/(\delta k)$ when $\gamma = 1$.  
For more general $\gamma$, the coalescent tree for any $\gamma$ can be simulated from rate-1 Yule process with time coordinate 
\citep[Eq.~(4.2)]{Ignatieva20} 
\[
u := \int_0^\tau \mu_{\rm eff}(\xi) d\xi = \log\{1 + \gamma^{-1}(e^{\delta \tau} - 1)\}, 
\]
and transforming back to the physical time coordinate via the inverse transformation $\delta \tau = \log\{1 + \gamma(e^u - 1)\}$.  Examples of coalescent tree 
shapes are shown in Fig.~\ref{fig:exampleTreesPlot}.  
\begin{figure}[t!]
\begin{center}
\centerline{\includegraphics[width=0.65\textwidth]{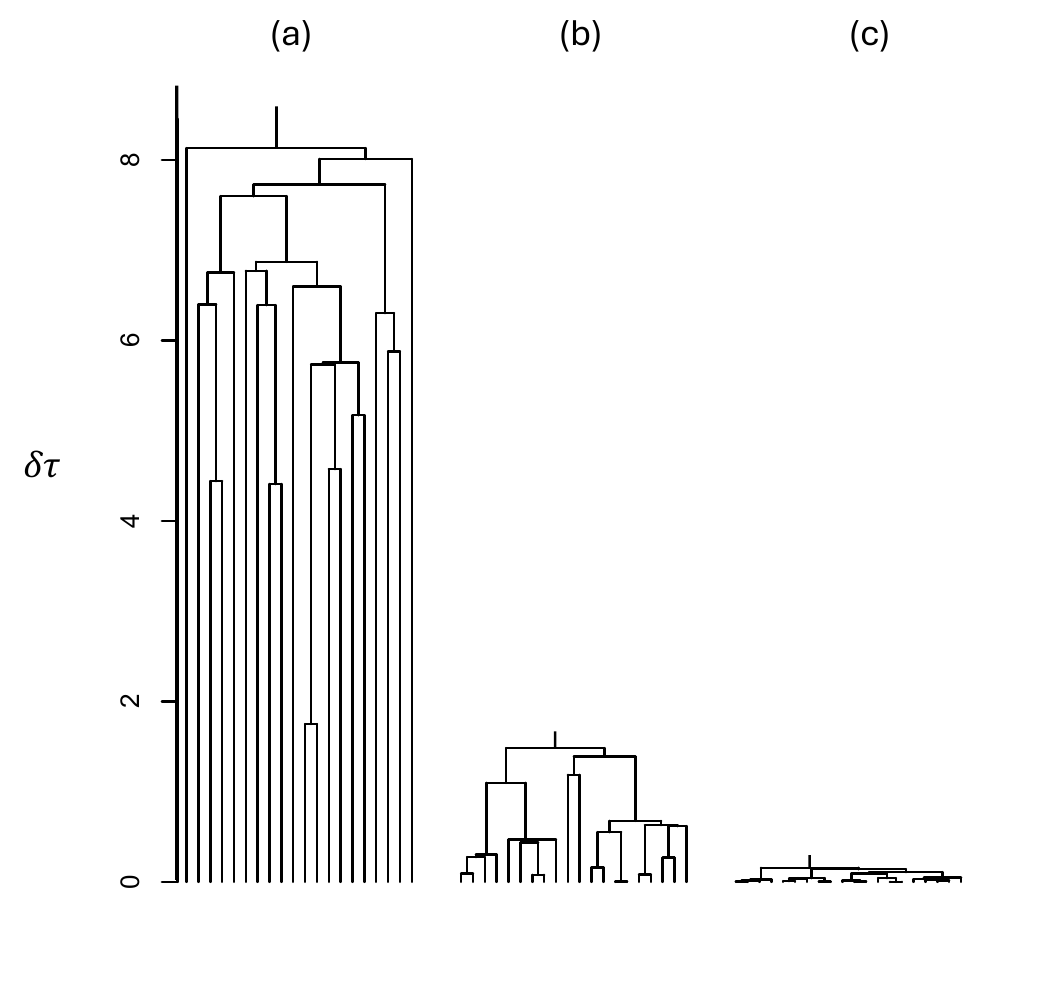}}
\caption{(a) The coalescent tree of a Bernoulli-sampled BD process or Poisson-sampled Feller diffusion conditioned on a sample of size 20 
for (a) $\gamma = 1000$, (b) $\gamma = 1$  (equivalent to a rate-1 Yule process), and (c) $\gamma = 0.05$.  } 
\label{fig:exampleTreesPlot}
\end{center}
\end{figure}

\citet[Section~4.3]{Ignatieva20} use the same inverse transformation from a rate-1 Yule process to calculate densities for the inter-coalescent times $W_k$.  
The coalescent times $T_k$ are shown in the limit $\gamma\to \infty$ to be equal in distribution to the order statistics of a logistic random variable 
up to a shift of $\delta^{-1}\log\gamma$.  This picture is consistent with the $\gamma \to \infty$ limit in Theorem~\ref{theorem:CasesWk}, which 
is relevant to supercritical processes sampled at low rates relative to population size.  From 
Eq.~(\ref{EofWgammaInf}) and the identity $T_j = \sum_{j = k}^n W_j, k = 1, \ldots, n,$ one obtains 
\[
\delta \mathbb{E}_k^{{\rm unif}\, [0, \infty)}[T_k] \sim \begin{cases}
	\log\gamma + \sum_{j = 1}^{n - 1} \frac{1}{j} 						,		& k = 1; \nonumber \\
	\log\gamma + \sum_{j = 1}^{n - k} \frac{1}{j} -  \sum_{j = 1}^{k - 1} \frac{1}{j}  	,		& k = 2, \ldots n - 1; \nonumber \\
	\log\gamma - \sum_{j = 1}^{n - 1} \frac{1}{j} 						,		& k = n, 
	\end{cases}
\]
as $\gamma\to \infty$, which agrees with \citet[Eqs.~(4.11) and (4.12)]{Ignatieva20}.  

%

\section{Conclusions}
\label{sec:Conclusions}

We have surveyed a number of published results related to scaling limits of linear BD  process and shown in each case the limit to be a manifestation 
of a Feller diffusion.  The unified approach taken has been via a limit of a CPP in which the node-height distribution becomes heavily skewed towards 
zero while keeping the stem height fixed, consequently allowing the number of leaves of the associated coalescent tree to tend towards infinity.  
The resulting coalescent tree acquires the familiar property of coming down from infinity, typical of diffusion limits in population genetics models.  
The approach allows for efficient derivation of properties of Feller diffusions, including the joint distribution of coalescent times when conditioning 
on an observed scaled population $X(t) = x$ in Theorem~\ref{thm:JointTkGivenX}, where $X(t)$ is the scaled population defined in 
Section~\ref{sec:FellerCPP}.  

One important property of branching processes associated with CPPs is that the reconstructed tree of the Bernoulli-sampled process is itself 
a CPP.  Analogously, we define in Subsection~\ref{sec:PSamplingFeller} a procedure called Poisson sampling in the diffusion limit.  Whereas 
for Bernoulli sampling with rate $\rho \in [0, 1]$ the sample size is binomial with success probability $\rho$ and number of trials equal 
to the population $N(t)$, for Poisson sampling with rate $\nu \in [0, \infty)$ the sample size is Poisson distributed with mean $2\nu X(t)$.  
In common with both Bernoulli-sampled and unsampled BD processes, the coalescent tree of a Poisson-sampled Feller diffusion is found to correspond 
to a CPP with node-height {\em cdf} taking the algebraic form of Eq.~(\ref{FHWiuf}), but with parameter $\gamma$ now defined by Eq.~(\ref{gammaNuDef}).   

\citet{Lambert18} has developed a procedure for determining coalescent properties for $k$-sampled populations once the corresponding properties 
of Bernoulli sampling have been determined.  In Subsection~\ref{sec:kSamplingFeller} we have extended the 
procedure to Poisson-sampled diffusions, and in Section~\ref{sec:kSampledT2toTk} apply it to the $k$-sampled Feller diffusion to rederive the joint 
distribution of coalescent times for a near-critical continuous-time Galton-Watson process previously found by \citet{Harris20}.  
In Section~\ref{sec:kSampledT2toTkXtConditioned} we demonstrate that the Limit Birth-Death process defined and analysed by \citet{Crespo21} is 
in fact a $k$-sampled Feller diffusion conditioned on a current observed population $X(t) = x$, and that the methods employed should in principle have 
further application to the $k$-sampling of diffusion processes.  

In Section~\ref{sec:ExpectedWkWiuf} we derive a formula in terms of hypergeometric functions for expected inter-coalescent waiting times of trees 
derived from CPPs associated with Wiuf's node-height distribution function Eq.~(\ref{FHWiuf}).  This functional form includes Bernoulli-sampled and 
unsampled BD processes and Poisson-sampled Feller diffusions.  

%

\section*{Acknowledgement} We thank two referees for a careful reading of the manuscript and their comments and suggestions which have improved the paper.

%

\appendix

%

\section{Notation}
\label{sec:notation}

\begin{tabular}{l p{10cm}}
$H_1, H_2, \ldots$							&Independent and identically distributed node heights in a CPP, as in Fig.~\ref{fig:CPP_diagram} \\
$t$										&Reserved for the stem age of a CPP, or equivalently, the time back to the root of the associated coalescent tree.  \\
$T_1, T_2, \ldots$							&Coalescent times as in Fig.~\ref{fig:CPP_diagram}.  $T_1 = t$ is the root.  
											$T_k, k = 2, 3, \,\ldots$ is the $k$-ancestor to $(k - 1)$-ancestor coalescent time.  \\ \\
$F_H(\tau)$, $f_H(\tau)$						&Typical node height {\em cdf} and density function for an unspecified CPP. \\
$F_H^{\rm IT}(\tau)$							&Node height inverse tail distribution, $F_H^{\rm IT}(\tau) = \mathbb{P}(H > \tau)^{-1}$.  \\
$F_H^{\rm BD}(\tau)$						&Node height {\em cdf} for a BD process, normalised to ensure convergence to a single ancestral founder as $t \to \infty$, 
											Eq.~(\ref{FHaltBD}), or equivalently Eq.~(\ref{FHdeltagamma}).\\
$F_{H_\rho}^{\rm BD}(\tau)$					&Node height {\em cdf} for a Bernoulli-sampled BD process sampled at rate $\rho$, normalised to ensure convergence 
											to a single ancestral founder as $t \to \infty$, Eq.~(\ref{FHrhoAltBD}).\\ 
$F_{H_\nu}^{\rm Feller}(\tau)$					&Node height {\em cdf} for a Poisson-sampled Feller process sampled at rate $\nu$, normalised to ensure convergence 
											to a single ancestral founder as $t \to \infty$, Eq.~(\ref{FHnuFeller}).\\
$F_H^{\rm Wiuf}(\tau)$						&In Section~\ref{sec:ExpectedWkWiuf}, any node height {\em cdf} taking the generic algebraic form of $F_H^{\rm BD}(\tau)$, 													$F_{H_\rho}^{\rm BD}(\tau)$ or $F_{H_\nu}^{\rm Feller}(\tau)$, viz.\ Eq.~(\ref{FHWiuf}).\\ \\
$\mu_{\rm eff}(\tau)$							&Instantaneous death-rate in a RRP.  \\
$\lambda, \mu$								&Instantaneous constant birth and death rates in a BD process. \\
$\rho$									&Bernoulli sampling rate a BD process.  \\
$\lambda_\epsilon, \mu_\epsilon, \rho_\epsilon$	&Instantaneous birth and death rates and Bernoulli sampling rate in the pre-limit BD process leading to 
											a Feller diffusion as $\epsilon \to 0$ (Sections~\ref{sec:FellerCPP} and \ref{sec:PSamplingFeller}). \\
$\alpha$									&Growth rate in a Feller process. \\
$\delta$									&For a BD process, $|\lambda - \mu|$ (Eq.~(\ref{WiufParams})); for a Feller process, 
											$\lim_{\epsilon\to 0}|\lambda_\epsilon - \mu_\epsilon| = |\alpha|$ (Eq.~(\ref{deltaGamma_epsilon})).  \\
$\gamma$								&The parameter defined by Eq.~(\ref{WiufParams}) occurring in $F_H^{\rm BD}(\tau)$.  Also used in 
											Section~\ref{sec:ExpectedWkWiuf} to stand for any of $\gamma$, $\gamma_\rho$ or $\gamma_\nu$.  \\
$\gamma_\rho$								&The parameter defined by Eq.~(\ref{WiufParamsRho}) occurring in $F_{H_\rho}^{\rm BD}(\tau)$. \\
$\gamma_\nu$								&The parameter defined by Eq.~(\ref{gammaNuDef}) occurring in $F_{H_\nu}^{\rm Feller}(\tau)$. 
\end{tabular}

%

\section{Iterative formulae for $\mathbb{E}_n[T_{k} \mid T_1 = t]$ and $\mathbb{E}_n[W_{k} \mid T_1 = t]$.}
\label{sec:iterativeTandW}

The expectation of the coalescent time $T_k$ conditioned on $T_1 = t$ and $n$ leaves is 
\[
\mathbb{E}_n[T_{k} \mid T_1 = t] = \int_0^t \tau f_{T_k \mid T_1 = t}(\tau \mid n) d\tau, 
\]
with $f_{T_k \mid T_1 = t}(\tau \mid n)$ as in Eq.~(\ref{Tk_density}).  Following \citet[Appendix~C]{Wiuf18} set 
\[
u = G^{-1}(\tau) := 1 - \frac{F_H(\tau)}{F_H(t)}, \quad du = -\frac{f_H(\tau)}{F_H(t)} d\tau. 
\]
Then
\begin{eqnarray} 
\lefteqn{\mathbb{E}_n[T_{k} \mid T_1 = t]} \nonumber \\
	& = & \frac{(n - 1)!}{(n - k)!(k - 2)!} \int_0^1 u^{k - 2} (1 - u)^{n - k}G(u) du \label{EnTkgivenT1} \\
	& = & \frac{(n - 1)!}{(n - k)!(k - 2)!} \left\{\int_0^1 u^{k - 3} (1 - u)^{n - k}G(u) du - 
							\int_0^1 u^{k - 3} (1 - u)^{n - k + 1}G(u) du \right\} \nonumber \\
	& = & \frac{n - 1}{k - 2} \frac{(n - 2)!}{(n - k)!(k - 3)!}  \int_0^1 u^{k - 3} (1 - u)^{n - k}G(u) du \nonumber \\
	&&		\qquad -\, \frac{n - k + 1}{k - 2} \frac{(n - 1)!}{(n - k + 1)!(k - 3)!} \int_0^1 u^{k - 3} (1 - u)^{n - k + 1}G(u) du \nonumber \\
	& = & \frac{n - 1}{k - 2} \mathbb{E}_{n - 1}[T_{k - 1} \mid T_1 = t] - \frac{n - k + 1}{k - 2} \mathbb{E}_n[T_{k - 1} \mid T_1 = t].   \nonumber
\end{eqnarray}

Eq.~(\ref{iterativeWkgivenT1}) for expected waiting times $W_k$ is confirmed by repeated use of Eq.~(\ref{iterativeTkgivenT1}).  The left hand side is 
\begin{eqnarray*}
\lefteqn{\mathbb{E}_n[T_k \mid T_1 = t] - \mathbb{E}_n[T_{k + 1} \mid T_1 = t]}\nonumber \\
	& = & \frac{n - 1}{k - 2} \mathbb{E}_{n - 1}[T_{k - 1} \mid T_1 = t] - \frac{n - k + 1}{k - 2} \mathbb{E}_n[T_{k - 1} \mid T_1 = t] \nonumber \\
	&&		-\, \frac{n - 1}{k - 1} \mathbb{E}_{n - 1}[T_{k} \mid T_1 = t] + \frac{n - k}{k - 1} \mathbb{E}_n[T_{k} \mid T_1 = t],  
\end{eqnarray*}
and the right hand side is 
\begin{eqnarray*}
	&  &\frac{n - 1}{k - 1} \mathbb{E}_{n - 1}[T_{k - 1} \mid T_1 = t] - \frac{n - 1}{k - 1} \mathbb{E}_{n - 1}[T_{k} \mid T_1 = t] \nonumber \\
	&  & -\, \frac{n - k + 1}{k - 1} \mathbb{E}_n[T_{k - 1} \mid T_1 = t] + \frac{n - k + 1}{k - 1} \mathbb{E}_n[T_{k} \mid T_1 = t].  	
\end{eqnarray*} 
The left hand side minus the right hand side is, after gathering terms and simplifying, 
\begin{eqnarray*}
\lefteqn{\frac{1}{k - 1}\left\{ \frac{n - 1}{k - 2}\mathbb{E}_{n - 1}[T_{k - 1} \mid T_1 = t] \right.} \nonumber \\  
		& & \qquad\qquad\qquad \left. -\,  \frac{n - k + 1}{k - 2} \mathbb{E}_n[T_{k - 1} \mid T_1 = t] - \mathbb{E}_n[T_{k} \mid T_1 = t]\right\} = 0, 
\end{eqnarray*} 
by Eq.~(\ref{iterativeTkgivenT1}), as required.

%

\section{Proof of Theorem~\ref{thm:JointTkGivenX}}
\label{sec:JointTkProof}

First note that the pre-limit parameters $\delta$ and $\gamma_\epsilon$ satisfy Eq.~(\ref{deltaGamma_epsilon}).  
Thus it is sufficient to prove the result for $|\alpha| = \delta \ge 0$.  Furthermore, taking into account the scaling  $X(t) = \epsilon N_\epsilon(t)$,
the condition $X(t) = x$ is effected by setting $\gamma_\epsilon = 2\delta x/n$ 
and taking the limit $n \to \infty$ with fixed $\delta$ in Eqs.~(\ref{T2_to_Tk_density}) and (\ref{Tk_density}), which are expressed in terms of the 
convention of Eq.~(\ref{normalisedFH}).  

The limiting behaviour of each factor is 
\[
\frac{(n - 1)!}{(n - k)!} \sim n^{k - 1}; 
\]
\begin{eqnarray*}
f(\tau_j) = \frac{f^{\rm BD}_H(\tau_j)}{F^{\rm BD}_H(t)} & = & \frac{2\delta^2 xe^{\delta \tau_j}}{n(e^{\delta \tau_j} - 1 + \frac{2\delta x}{n})^2} 
							\times \frac{e^{\delta t} - 1 + \frac{2\delta x}{n}}{e^{\delta t} - 1} \nonumber \\
		& \sim & \frac{2\delta^2 x e^{\alpha \tau_j}}{n(e^{\alpha\tau_j} - 1)^2} \times 1 \nonumber \\ 
		& = & \frac{1}{n} \frac{x \mu(\tau_j; |\alpha|)}{2\beta(\tau_j; |\alpha|)} ; 
\end{eqnarray*}
\[
F^{\rm BD}_H(\tau)^{n - k} \sim \left(1 - \frac{2\delta x}{n(e^{\delta \tau} - 1)}\right)^{n - k} \sim e^{-x/\beta(\tau; \delta)}, 
\]
leading to 
\[
F(\tau_k)^{n - k} = \left(\frac{F^{\rm BD}_H(\tau_k)}{F^{\rm BD}_H(t)}\right)^{n - k} \sim e^{-(x/\beta(\tau_k; \delta) - x/\beta(t; \delta))};   
\]
and 
\begin{eqnarray*}
1 - F(\tau) & = & 1 - \frac{F^{\rm BD}_H(\tau)}{F^{\rm BD}_H(t)}  \nonumber \\ 
	& \sim & 1 - \left(1 - \frac{2\delta x}{n(e^{\delta \tau} - 1)}\right) \left(1 - \frac{2\delta x}{n(e^{\delta t} - 1)}\right)^{-1} \nonumber \\ 
	& \sim & 1 - \left(1 - \frac{1}{n}\frac{x}{\beta(\tau; \delta)}\right) \left(1 - \frac{1}{n}\frac{x}{\beta(t; \delta)}\right)^{-1} \nonumber \\ 
	& \sim & \frac{1}{n}\left(\frac{x}{\beta(\tau; \delta)} - \frac{x}{\beta(t; \delta)}\right). 
\end{eqnarray*}
Assembling the factors leads to Eqs.~(\ref{jointTkGivenx}) and (\ref{TkGivenx}).  
\qed

%

\section{Proof of Theorem~\ref{thm:BernoulliToPoisson}}
\label{sec:proofBToP}

Begin with a pre-limit BD process with 
node height {\em cdf} from Eq.~(\ref{FHdeltagamma}), 
\[
a_\epsilon :=  \frac{e^{\delta\tau} - 1}{e^{\delta\tau} - 1 + \gamma_\epsilon}, \qquad \tau, \gamma_\epsilon, \delta > 0,   
\]
where $\epsilon$ is the parameter introduced in Section~\ref{sec:FellerCPP} for defining the Feller diffusion limit, 
and $\delta$ and $\gamma_\epsilon$ are defined by Eq.~(\ref{deltaGamma_epsilon}).  In Section~\ref{sec:PSamplingFeller}, 
Poisson sampling at rate $\nu$ of a Feller diffusion is defined as the limit as $\epsilon \to 0$ of Bernoulli sampling 
at rate $\rho_\epsilon = 2\nu\epsilon$ of the pre-limit BD process.  

The measure for Poisson to $k$-sampling of a Feller diffusion analogous to Eq.~(\ref{muKDRhoDef}) is 
\[
\mu_k(d\nu) := \lim_{\epsilon \to 0} \frac{k(1 - a_\epsilon)\rho_\epsilon^{k - 1} d\rho_\epsilon}{(1 - a_\epsilon(1 - \rho_\epsilon))^{k + 1}} .  
\]
The numerator is 
\[
k \frac{2\delta\epsilon}{e^{\delta t} - 1 + 2\delta\epsilon} (2\nu\epsilon)^{k - 1}   d(2\epsilon\nu) = 
			\frac{2^{k + 1} k\delta \nu^{k - 1} d\nu}{e^{\delta t} - 1} ( \epsilon^{k + 1} + \mathcal{O}(\epsilon^{k + 2})),  
\]
and the denominator is 
\[
\left(1 - \frac{e^{\delta t} - 1}{e^{\delta t} - 1 + 2\delta\epsilon} (1 - 2\nu\epsilon) \right)^{k + 1} =  
		\left( \frac{2(\delta + \nu(e^{\delta t} - 1)}{e^{\delta t} - 1} \right)^{k + 1}  ( \epsilon^{k + 1} + \mathcal{O}(\epsilon^{k + 2})).
\]
Thus the required measure is 
\[
\mu_k(d\nu) = \frac{k\delta \nu^{k - 1} (e^{\delta t} - 1)^k}{(\delta + \nu(e^{\delta t} - 1))^{k + 1}} d\nu = \frac{k v^{k - 1}}{(1 + v)^{k + 1}}  dv
		, \quad v \in (0, \infty),  
\]
where
\[
v = \frac{\nu}{\delta} (e^{\delta t} - 1) = 2\nu\beta(t; \delta),  
\]
agreeing with Eq.~(\ref{mukOfDnu}).
\qed

%

\section{Proof of Theorem~\ref{AltHarrisTheorem}}
\label{sec:HarrisProof}

Applying Theorem~\ref{thm:BernoulliToPoisson}, the joint {\em cdf} of coalescent times $T_2, \ldots, T_k$ under $k$ sampling is 
\begin{eqnarray} 	\label{cdfTkFromcdfBern}
\lefteqn{F_{T_2\ldots T_k}^{k\text{-}{\rm sample}}(\tau_2, \ldots, \tau_k \mid T_1 = t) 
	:= \mathbb{P}(T_2 \le \tau_2, \ldots, T_k \le \tau_k \mid T_1 = t)}  \nonumber \\
	& = & \int_0^\infty \frac{k v^{k - 1}}{(1 + v)^{k + 1}}  
					F_{T_2\ldots T_k}^{{\rm Poisson}\,\nu}(\tau_2, \ldots, \tau_k \mid T_1 = t) dv.
\end{eqnarray}
Since the Poisson-sampled coalescent tree is generated by a CPP, we can apply Eq.~(\ref{T2_to_Tn_cdf}) to write 
\[
F_{T_2\ldots T_k}^{{\rm Poisson}\,\nu}(\tau_2, \ldots, \tau_k \mid T_1 = t) = (k - 1)! \prod_{j = 2}^k \frac{F_{H_\nu}(\tau_j)}{F_{H_\nu}(t)}
\]
where $F_{H_\nu}(t)$ is the node height {\em cdf} Eq.~(\ref{FHnuAltFeller}).  Then 
\begin{eqnarray*}
\frac{F_{H_\nu}(\tau)}{F_{H_\nu}(t)} & = & \frac{2\beta(\tau; \alpha)}{2\beta(\tau; \alpha) +1/\nu} \cdot \frac{2\beta(t; \alpha) +1/\nu}{2\beta(t; \alpha)} \nonumber \\
	& = & \frac{\beta(\tau; \alpha)}{\beta(\tau; \alpha) + \beta(t; \alpha)/v} \cdot \frac{\beta(t; \alpha) + \beta(t; \alpha)/v}{\beta(t; \alpha)} \nonumber \\
	& = & \frac{v + 1}{v + B(\tau)},
\end{eqnarray*}  
where $B(\tau) = \beta(t, \alpha)/\beta(\tau; \alpha)$ and $\beta(\tau; \alpha)$ is defined by Eq.~(\ref{betaDef}).  Substituting back into Eq.~(\ref{cdfTkFromcdfBern}), 
\begin{eqnarray*}
F_{T_2\ldots T_k}^{k\text{-}{\rm sample}}(\tau_2, \ldots, \tau_k \mid T_1 = t) 
	& = & k! \int_0^\infty \frac{v^{k - 1}}{(1 + v)^2} \left(\prod_{j = 2}^k \frac{1}{v + B(\tau_j)} \right) dv \nonumber \\
	& = & k! \int_0^\infty \left\{ \frac{C}{(1 + v)^2} + \frac{D}{1 + v} + \sum_{j = 2}^k \frac{G_j}{v + B(\tau_j)} \right\} dv, 
\end{eqnarray*}
where the partial fraction coefficients are determined from 
\[
v^{k - 1} = C \prod_{j = 2}^k (v + B(\tau_j)) + D(1 + v) \prod_{j = 2}^k (v + B(\tau_j)) + (1 + v)^2 \sum_{j = 2}^k G_j \prod_{l \ne j} (v + B(\tau_l)).  
\]
Setting $v = -1$ gives 
\[
C = \prod_{j = 2}^k (1 - B(\tau_j))^{-1} = \prod_{j = 2}^k \frac{\beta(\tau_j; \alpha)}{\beta(\tau_j; \alpha) - \beta(t; \alpha)}, 
\]
setting $v = -B(\tau_i)$ gives 
\begin{eqnarray*}
G_i & = & \frac{- B(\tau_i)^{k - 1}}{(1 - B(\tau_i))^2} \prod_{l = 2:k, \, l \ne i} (B(\tau_i) - B(\tau_l))^{-1} \nonumber \\
		& = & - \frac{\beta(\tau_j; \alpha)\beta(t; \alpha)}{(\beta(\tau_j; \alpha) - \beta(t; \alpha))^2} 
							\prod_{l = 2:k, \, l \ne i} \frac{\beta(\tau_l; \alpha)}{\beta(\tau_l; \alpha) - \beta(t; \alpha)},
\end{eqnarray*}
and equating coefficients of $v^k$ gives 
\[
D = -\sum_{j = 2}^k G_j. 
\]
Then 
\begin{eqnarray*}
F_{T_2\ldots T_k}^{k\text{-}{\rm sample}}(\tau_2, \ldots, \tau_k \mid T_1 = t) 
	& = & k! \int_0^\infty \left\{ \frac{C}{(1 + v)^2} + \sum_{j = 2}^k G_j \left(\frac{1}{v + B(\tau_j)} - \frac{1}{1 + v} \right) \right\} dv \nonumber \\ 
	& = & k! \left( C - \sum_{j = 2}^k G_j \log B(\tau_j) \right),   
\end{eqnarray*}
with $B(\tau_j), C$ and $G_j$ defined above and $t  >  \tau_2 > \ldots > \tau_k > 0$. 
 \qed

%

\section{Matching $k$-sample coalescent times to population coalescent times}
\label{sec:matchingTtildeT}

In Section~\ref{sec:kSampledT2toTkXtConditioned} formulae are given for the probability distribution of coalescent times for a uniform 
random sample of size $k$ among the coalescent times of an infinite population.  These formulae are adapted from analogous equations 
in \citet{Crespo21}, who in turn, make use of general results for subsampled exchangeable binary coalescent trees by \citet{Saunders84}.  

Figure~\ref{fig:m_kDef} shows part of a population coalescent tree in black, in which is embedded the corresponding part of the sample coalescent tree, 
shown in red.  For each of the sample coalescent times $\tilde{T}_l$, $2 \le l \le k$, \citet{Crespo21} set the (random) 
number of population ancestors immediately above the coalescent time in the diagram to $m_{l - 1}$.  Thus the event $m_{l - 1} = y$
is equivalent to the event $\tilde{T}_l = T_{y + 1}$.  Note also the boundary condition $m_k = \infty$ corresponding to $\tilde{T}_{k + 1} = T_\infty = 0$ 
at the leaves of the tree.  

\begin{figure}
 \centering
 \includegraphics[width=0.5\linewidth]{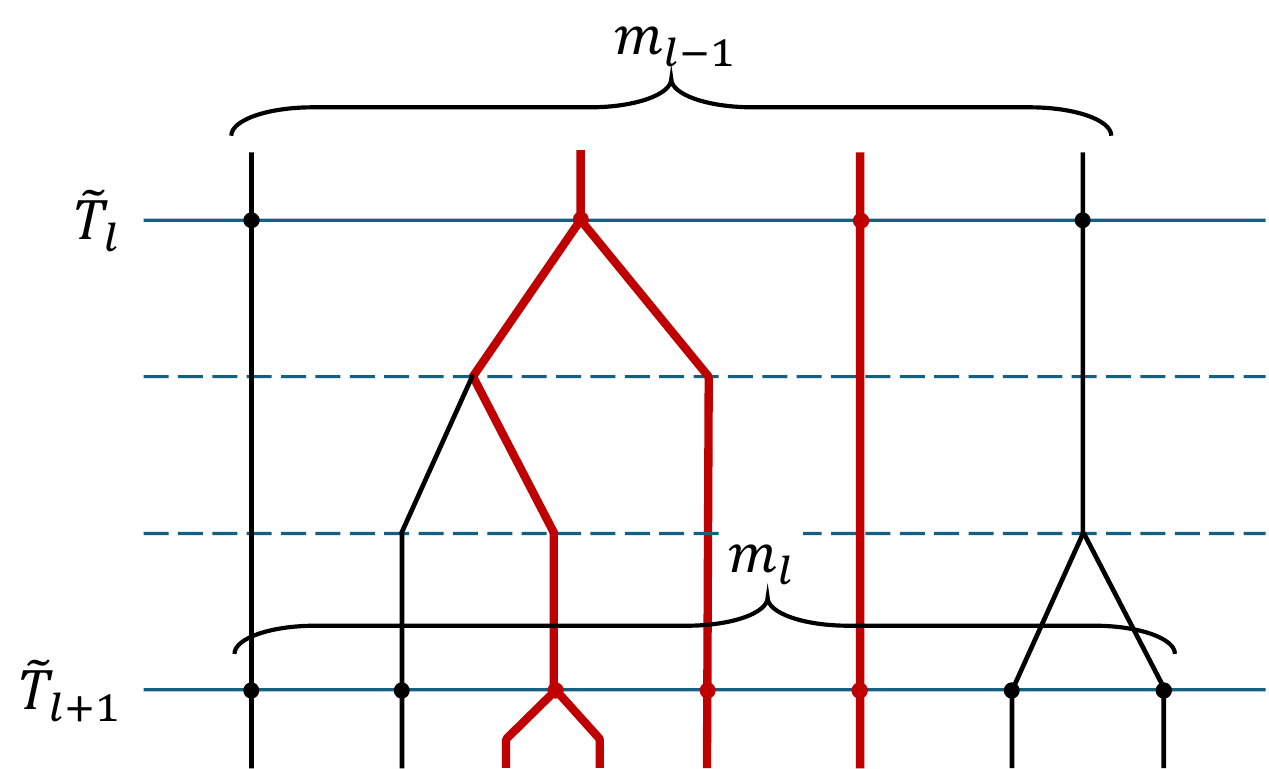}
  \caption{Section of a population coalescent tree in black and the embedded sample coalescent tree in red.  In this example the 
  sample coalescent times $\tilde{T}_l$ and $\tilde{T}_{l + 1}$ are shown for $l = 3$.  We have $m_{l - 1} = 4, m_l = 7$, and hence 
  $\tilde{T}_3 = T_5$ and $\tilde{T}_4 = T_8$. }
 \label{fig:m_kDef}
 \end{figure}

The {\em cdf}s for the $m_l$ are stated iteratively in \citet[Eqs.~(7) and (6)]{Crespo21} as 
\begin{eqnarray*}
\mathbb{P}(m_{k - 1} \le y \mid m_k = \infty) & = & \frac{y!}{(y - k + 1)!}\frac{(y + 1)!}{(y + k)!}, \qquad y = k - 1, k, \ldots, \nonumber \\
\mathbb{P}(m_{l - 1} \le y \mid m_l) & = & \frac{(m_l - l)!(m_l + l - 1)!}{m_l!(m_l - 1)!} \frac{y!}{(y - l + 1)!}\frac{(y + 1)!}{(y + l)!},  \nonumber \\
	& & \qquad\qquad y = l - 1, \ldots, m_l - 1; \, l = 2, \ldots, k - 1,
\end{eqnarray*}
from which one obtains the probabilities\footnote{Alternatively, start from \citet[Lemma~3]{Saunders84} with the replacements 
$i \rightarrow m_l, l \rightarrow y, j \rightarrow l, k \rightarrow l - 1$.}
\begin{eqnarray*}
\mathbb{P}(m_{k - 1} = y \mid m_k = \infty) & = & k(k - 1) \frac{y!(y - 1)!}{(y - k + 1)!(y + k)!}, \qquad y = k - 1, k, \ldots, \nonumber \\
\mathbb{P}(m_{l - 1} = y \mid m_l) & = & l(l - 1) \frac{(m_l - l)!(m_l + l - 1)!}{m_l!(m_l - 1)!} \frac{y!(y - 1)!}{(y - l + 1)!(y + l)!} \nonumber \\
	& & \qquad\qquad y = l - 1, \ldots, m_l - 1; \, l = 2, \ldots, k - 1.
\end{eqnarray*}
Eqs.~(\ref{PTtildenEqualsT}) and (\ref{PTtildekEqualsT}) then follow from 
\begin{equation*}
\begin{split}
\mathbb{P}(\tilde{T}_k = T_i) & = \mathbb{P}(m_{k - 1} = i - 1 \mid m_k = \infty) \nonumber \\ 
\mathbb{P}(\tilde{T}_l = T_i \mid \tilde{T}_{l + 1} = T_j) & = \mathbb{P}(m_{l - 1} = i - 1 \mid m_l = j - 1). 
\end{split}
\end{equation*}

%

\section{Calculation of $\mathbb{E}_n^{{\rm unif}\, [0, \infty)}[W_1]$}
\label{sec:EnW1Calc}

It is not immediately clear that the derivation of $\mathbb{E}_n^{{\rm unif}\, [0, \infty)}[W_1]$ in \citet{Wiuf18}
depends only on the form Eq.~(\ref{FHWiuf}) of the node-height distribution, as required in the current paper.   This appendix fills in the details.  
From Eqs.~(\ref{EgivenT1toEunif}) and (\ref{EnTkgivenT1}) we have 
\[
\mathbb{E}_n^{{\rm unif}\, [0, \infty)}[T_1] = \mathbb{E}_{n + 1}[T_{2} \mid T_1 = \infty] = n\int_0^1 (1 - u)^{n - 1} G(u) du, 
\]
where, from Eq.~(\ref{FHWiuf}), $G^{-1}(\tau) = 1 - F_H^{\rm Wiuf}(\tau) = \gamma/(e^{\delta \tau} - 1 + \gamma)$, which inverts to  
\[
G(u) = \frac{1}{\delta} \log\frac{\gamma + (1 - \gamma)u}{u}, 
\]
Hence 
\[
\mathbb{E}_n^{{\rm unif}\, [0, \infty)}[T_1] = \frac{n}{\delta} \int_0^1 (1 - u)^{n - 1} \log\frac{\gamma + (1 - \gamma)u}{u} du.
\]
This is the starting point in \citet[Appendix E]{Wiuf18}, with the expected waiting time Eq.~(\ref{EnOfW1}) then derived in 
\citet[Appendix G]{Wiuf18}.  

%

\section{Proof of Theorem~\ref{theorem:ExpectedWk}}
\label{sec:proofExpectedWk}

By using the identity 
\[
\log\gamma = \log[1 - (1 - \gamma)] = - \sum_{i = 1}^\infty \frac{(1 - \gamma)^i}{i}, \qquad |1 - \gamma| < 1, 
\]
Eq.(\ref{EnOfW1}) can be rewritten for $|1 - \gamma| < 1$ as 
\[
\mathbb{E}_n^{{\rm unif}\, [0, \infty)}[W_1] = \frac{\gamma}{\delta} \sum_{l = 0}^\infty \frac{n}{n + l} (1 - \gamma)^l.  
\]
For $k = 1, \ldots, n$, this generalises to the solution of the recursion Eq.~(\ref{iterativeWkunifT1}), 
\begin{equation}	\label{EnOfWkInfSum}
\mathbb{E}_n^{{\rm unif}\, [0, \infty)}[W_k] = \frac{\gamma}{\delta} \frac{1}{k} \sum_{l = 0}^\infty \frac{n_{[k]}}{(n + l)_{[k]}} (1 - \gamma)^l, 
			\qquad |1 - \gamma| < 1, 
\end{equation}
as can be confirmed by induction.  Here, $n_{[k]} = n(n - 1) \cdots (n - k + 1)$ is the falling factorial.  Eq.~(\ref{EnOfWkUnifT1}) follows 
immediately for $|1 - \gamma| < 1$ using Eq.~(\ref{HypergeomDef}).  The hypergeometric function Eq.~(\ref{HypergeomDef}) can be analytically continued 
via \citet[Eq.~(15.3.1)]{Abramowitz:1965sf} throughout the complex $z$-plane cut along the real axis from $1$ to $\infty$, and hence 
Eq.~(\ref{EnOfWkUnifT1}) is defined for $\gamma > 0$ by analytic continuation.  
\qed

%

\section{Proof of Theorem~\ref{theorem:CasesWk}}
\label{sec:proofCasesWk}

The $k = 1$ case of Eq.~(\ref{EofWgamma0}) follows immediately from Eq.~(\ref{EnOfW1}).  For $k = 2, \ldots, n$, begin with Eq.~(\ref{EnOfWkUnifT1}).  \citet[Eq.~(15.1.20)]{Abramowitz:1965sf} gives
\[
_2F_1(n - k + 1, 1; n + 1; 1) = \frac{n}{k - 1}, 
\]
and the remaining cases of Eq.~(\ref{EofWgamma0}) follow.    

From Eq.~(\ref{HypergeomDef}), $_2F_1(a, b; c; 0)) = 1$, from which Eq.~(\ref{EofWgamma1}) immediately follows.  

From \citet[Eq.~(15.3.5)]{Abramowitz:1965sf}, 
\[
_2F_1(n - k + 1, 1; n + 1; 1 - \gamma) = {\gamma^{-1}} \,_2F_1(1, k; n + 1; 1 - \gamma^{-1}).  
\]
Then for $k = 1, \ldots, n - 1$, using \citet[Eq.~(15.1.20)]{Abramowitz:1965sf}, 
\[
\lim_{\gamma \to \infty} \mathbb{E}_n^{{\rm unif}\, [0, \infty)}[W_k] = \frac{1}{\delta k} \,_2F_1(1, k; n + 1; 1) = \frac{1}{\delta k} \frac{n}{n - k}, 
\]
and the first line of Eq.~(\ref{EofWgammaInf}) follows.    
\citet[Section~8.3 and Appendix~H]{Wiuf18}  calculates the expected total length of the tree to be 
\begin{equation}	\label{sumT_i}
\mathbb{E}_n^{{\rm unif}\, [0, \infty)}\left[\sum_{i = 1}^n i W_i \right] = \mathbb{E}_n^{{\rm unif}\, [0, \infty)}\left[\sum_{i = 1}^n T_i \right] = 
								\frac{n}{\delta} \frac{\gamma}{\gamma - 1} \log\gamma, \qquad \gamma \ne 1. 
\end{equation}
Hence
\begin{eqnarray*}
n \mathbb{E}_n^{{\rm unif}\, [0, \infty)}[W_n] & = & \frac{n}{\delta} \frac{\gamma}{\gamma - 1} \log\gamma - \sum_{i = 1}^{n - 1}  i \mathbb{E}_n\left[W_i \right] \\
	& = & \frac{n}{\delta} \frac{\gamma}{\gamma - 1} \log\gamma - \sum_{i = 1}^{n - 1} \frac{n}{\delta(n - i)} + \mathcal{O}(\gamma^{-1})\\
	& = & n \left\{ \frac{1}{\delta} \left( \log\gamma - \sum_{i = 1}^{n - 1} \frac{1}{i} \right) \right\}  + 
														\mathcal{O}(\gamma^{-1}) \quad \text{as } \gamma\to\infty,  
\end{eqnarray*}
confirming the second line of Eq.~(\ref{EofWgammaInf}).  
\qed

%

\bibliographystyle{elsarticle-harv}\biboptions{authoryear}

\end{document}